\documentclass[10pt,oneside]{amsart}
\usepackage[utf8]{inputenc}
\usepackage{mathtools}
\usepackage{amsmath}
\usepackage{amsfonts}
\usepackage{amssymb,amsthm}
\usepackage{enumitem}
\usepackage{microtype}
\usepackage{graphicx}
\usepackage{grffile}
\usepackage{color}
\usepackage[textheight=8in, textwidth=5in]{geometry}
\usepackage{hyperref}
\usepackage{algorithm}
\usepackage{algpseudocode}

\makeatletter
\g@addto@macro\bfseries{\boldmath}
\makeatother

\makeatletter
\renewcommand*\env@matrix[1][*\c@MaxMatrixCols c]{%
  \hskip -\arraycolsep
  \let\@ifnextchar\new@ifnextchar
  \array{#1}}
\makeatother

\renewcommand{\epsilon}{\eps}
\newcommand{\ttimes}{\!\times\!}

\usepackage{ dsfont }
 \newcommand{\one}{\mathds{1}}

\newcommand{\setmid}{\;:\;}

\DeclareMathOperator{\TO}{\mc L}

\DeclareMathOperator{\Isom}{Isom}

\DeclareMathOperator{\SL}{SL}
\DeclareMathOperator{\PSL}{PSL}

\DeclareMathOperator{\diag}{diag}

\DeclareMathOperator{\Tr}{Tr}
\DeclareMathOperator{\tr}{tr}

\DeclareMathOperator{\Ima}{Im}
\DeclareMathOperator{\Rea}{Re}

\newcommand\N{\mathbb{N}}

\newcommand\R{\mathbb{R}}

\newcommand\C{\mathbb{C}}

\newcommand{\h}{\mathbb{H}}

\newcommand{\mc}[1]{\mathcal #1}

\newcommand{\wh}{\widehat}

\newcommand{\eps}{\varepsilon}

\DeclareMathOperator{\Fct}{Fct}

\newcommand{\sceq}{\mathrel{\mathop:}=}

\newcommand{\bmat}[4]{\begin{bmatrix} #1&#2\\#3&#4\end{bmatrix}}
\newcommand{\textmat}[4]{\left(\begin{smallmatrix} #1&#2 \\ #3&#4
\end{smallmatrix}\right)}
\newcommand{\textbmat}[4]{\left[\begin{smallmatrix} #1&#2 \\ #3&#4
\end{smallmatrix}\right]}

\theoremstyle{plain}
\newtheorem{prop}{Proposition}[section]

\newtheorem{thm}[prop]{Theorem}

\newtheorem*{obsnn}{Observation}

\theoremstyle{definition}

\newtheorem{example}[prop]{Example}

\theoremstyle{remark}

\begin{document}

\title[Numerical resonances]{Numerical resonances for Schottky surfaces via Lagrange--Chebyshev approximation}

\author[O.F.~Bandtlow]{Oscar F.~Bandtlow}

\address{%
Oscar F.~Bandtlow\\
School of Mathematical Sciences\\
Queen Mary University of London\\
London E3 4NS\\
UK.
}
\email{o.bandtlow@qmul.ac.uk}

\author[A.\@ Pohl]{Anke Pohl}
\address{Anke Pohl, University of Bremen, Department 3 -- Mathematics, Bibliothekstr.\@
5, 28359 Bremen, Germany}
\email{apohl@uni-bremen.de}

\author[T.\@ Schick]{Torben Schick}
\address{Torben Schick, Hauburgsteinweg 47, 61476 Kronberg, Germany}
\email{torben.schick@adesso.de}

\author[A.\@ Wei{\ss}e]{Alexander Wei{\ss}e}
\address{Alexander Wei{\ss}e, Max Planck Institute for Mathematics, Vivatsgasse 7, 53111 Bonn, Germany}
\email{weisse@mpim-bonn.mpg.de}


\dedicatory{Dedicated to Manfred Denker on the occasion of his 75th birthday}

\begin{abstract} 
We present a numerical method to calculate resonances of Schottky surfaces based on Selberg theory, transfer operator techniques and Lagrange--Chebyshev approximation. This method is an alternative to the method based on periodic orbit expansion used previously in this context. 
\end{abstract}

\maketitle

\section{Introduction and statement of results}

In numerical experiments, Borthwick noticed that 
for some classes of Schottky surfaces the sets of resonances exhibit rather curious and unexpected patterns~\cite{Borthwick_experimental}, an observation that triggered a series of further investigations, e.g., \cite{Borthwick_Weich, Pollicott_Vytnova}. The numerics used in these articles is based on the method of periodic orbit expansion. With this article we present an alternative method, based on Lagrange--Chebyshev approximation with adapted domains.  
We briefly survey both methods and our main results in this section, referring to the next sections for details and precise definitions.

Let $X$ be a Schottky surface, i.e., a hyperbolic surface whose fundamental group~$\Gamma$ is a (Fuchsian) Schottky group. The resonances of~$X$ are the poles of the extension of the resolvent
\[
 R(s) \sceq \big(\Delta_X - s(1-s)\big)^{-1}\colon C_c^\infty(X) \to C^\infty(X),\quad \Rea s \gg 1\,,
\]
of the Laplacian~$\Delta_X$ of~$X$ to a meromorphic family $s\mapsto R(s)$ on all of~$\C$. As is well-known by now, the Hausdorff dimension~$\delta$ of the limit set~$\Lambda(X)$ of~$X$ constitutes a threshold in~$\C$ with respect to the location of the resonances of~$X$: the half-space $\{s\in\C\mid \Rea s>\delta\}$ does not contain any resonances, the point~$\delta$ itself is a resonance (with multiplicity~$1$), the \emph{critical axis}~$\{s\in\C\mid\Rea s = \delta\}$ contains no further resonance, the half-space $\{s\in\C\mid \Rea s < \delta\}$ contains infinitely many resonances, and $\{ \Rea s \mid \text{$s$ a resonance}\}$ is unbounded from below. 
Several further results on the location and distribution of the resonances for Schottky surfaces have been established, some of them very recently, see, e.g.,~\cite{GZ_upper_bounds, GZ_scattering, Bourgain_Dyatlov,Pohl_Soares, Naud_Magee}. 

The study of resonances of Schottky surfaces is not just interesting from a theoretical point of view, their properties are also of considerable interest in applications, e.g.,~\cite{BGS, Oh_Winter}. 
In spite of intensive efforts, however, many questions concerning the nature of the resonances are still wide open to date, including several very elementary ones. 
Numerical investigations of the resonance sets might lead to new insights, and are therefore clearly called for. In addition, as we will see in this paper, their investigation is also rewarding from a numerical point of view.

Both numerical methods that we will discuss crucially rely on the interpretation of resonances using dynamical zeta functions and transfer operators. The dynamical zeta function of interest to us is the Selberg zeta function~$Z_X$ of~$X$, which is an entire function given by the infinite product 
\[
 Z_X(s) = \prod_{\ell\in L(X)}\prod_{k=0}^\infty \left( 1 - e^{-(s+k)\ell} \right)
\]
for $s\in\C$ with $\Rea s\gg 1$, and by analytic continuation of this infinite product on all of~$\C$. Here, $L(X)$ denotes the primitive geodesic length spectrum of~$X$ including multiplicities. The resonances of~$X$ constitute the main bulk of the zeros of~$Z_X$ (including multiplicities); the additional zeros of~$Z_X$ are the well-understood and comparatively sparse so-called topological zeros~\cite{Patterson_Perry}. Thus, searching for the resonances of~$X$ is essentially equivalent to searching for the zeros of~$Z_X$. 

The Selberg zeta function can be represented as the Fredholm determinant of a transfer operator family~$(\TO_s)_{s\in\C}$, acting on a suitable function space~$\mc H$. Thus, for all~$s\in\C$,
\begin{equation}\label{eq:ZTO_intro}
 Z_X(s) = \det(1-\TO_s)\,.
\end{equation}
(We refer to Section~\ref{sec:def_TO} for the definition of the transfer operators~$\TO_s$.) In turn, searching for the zeros of~$Z_X$, and hence essentially for the resonances of~$X$, translates to searching for the values of the parameter~$s$ of the transfer operator~$\TO_s$ such that $\det(1-\TO_s)=0$.  

The \emph{method of periodic orbit expansion}, employed by Borthwick for the numerical approximation of resonances, takes advantage of the identity in~\eqref{eq:ZTO_intro} to deduce the series expansion
\begin{equation}\label{eq:poe_intro}
 Z_X(s) = 1+\sum_{n=1}^\infty d_n(s)\,,
\end{equation}
where $d_0(s) = 1$ and 
\begin{equation}\label{eq:dn}
 d_n(s) = - \frac1n \sum_{k=1}^n d_{n-k}(s)\Tr\TO_s^k 
\end{equation}
for $n\in\N$. The traces $\Tr\TO_s^k$ are of the form 
\begin{equation}\label{eq:trace_k}
 \Tr\TO_s^k = \sum_{\ell\in L_k(X)} \frac{e^{-s\ell}}{1-e^{-\ell}}\,,
\end{equation}
where $L_k(X)$ is a finite multiset of lengths of periodic geodesics on~$X$, with lengths growing with $k$, and $\#L_k(X)\sim c^k$ for some $c>0$. Truncating the series in~\eqref{eq:poe_intro} at $N\in\N$ leads to 
\[
 Z_N(s) = 1 + \sum_{n=1}^N d_n(s)\,,
\]
the zeros of which approximate the zeros of~$Z_X$.

Borthwick's primary goal was to study the resonances of~$X$ near the critical 
axis~$\{s\in\C\mid\Rea s=\delta\}$, for which the periodic orbit expansion method is well-suited. In addition to the rather surprising results he obtained, he also reported limitations of this method for the investigation of resonances  further away from the critical axis (e.g., resonances~$s$ with $\Rea s<0$) as well as for Schottky surfaces with Euler characteristic~$\chi(X)<-1$ 
(equivalent to $\Gamma$ having more than $2$ generators) or with small funnel widths (i.e., for $\delta$ near $1$), see~\cite{Borthwick_experimental}. The main cause of the 
limitations is the high computational cost of the recursions in~\eqref{eq:dn} as well as that of 
the sums in~\eqref{eq:trace_k}, the number of summands of which grow exponentially with~$k$ and polynomially with $-\chi(X)$, and which need to be almost completely redone for every single~$s$.

For \emph{highly symmetric} Schottky surfaces, the Venkov--Zograf factorization formulas for Selberg zeta functions can be applied to reduce the computational cost of the periodic orbit expansion method making it possible to consider, for example, some Schottky surfaces with $4$ funnels, that is, with Euler characteristic $-2$, see~\cite{Borthwick_Weich}. Qualitatively, however, the limitations remain. 

The method we present here is inspired by Nystr\"om's approach~\cite{Nystroem30} to solve Fredholm-type integral equations. The definition of the transfer operator $\TO_s$ introduced above leaves room for possible choices of the function space~$\mc H$.
While the periodic orbit expansion method is not affected by the choice of a suitable~$\mc H$, 
our method heavily depends on it; in fact, it takes advantage of this option, reconsidering the choice during investigations and adapting it to specific situations. 

The starting point of this method is a careful first choice of the function space~$\mc H$ on which we  consider the transfer operator~$\TO_s$ and approximate the functions $f\in\mc H$ using Lagrange--Chebyshev interpolation. On approximated functions, $\TO_s$ then acts as a \emph{finite} matrix $M_s$, and $\det(1-M_s)$ serves as an approximation of the 
Selberg zeta function $Z_X(s)$ at $s\in\C$. The evaluation of $\det(1-M_s)$ in turn
is facilitated by the existence of highly optimized algorithms for the calculation of determinants. Moreover, the specific structure of $M_s$ (see Section~\ref{sec:LC}) makes it possible to reuse a substantial part of previously computed values when changing~$s$.

To further improve the approach we take advantage of the possibility to change~$\mc H$ and pick a sequence of function spaces~$(\mc H_m)_{m\in\N}$ such that along any sequence~$(f_m)_{m\in\N}$ with $f_m\in\mc H_m$, $m\in\N$, the domains of the functions shrink and converge to the limit set~$\Lambda(X)$ of~$X$. By using a space of this sequence with sufficiently refined domains, we can take advantage of the geometry and dynamics of~$X$ to investigate the resonances of~$X$ over a wider range of parameters.

This method, termed \emph{domain-refined Lagrange--Chebyshev approximation}, pushes further the frontier of numerical investigations of resonances for Schottky surfaces.

\begin{obsnn}\label{thm:main}
With the method of domain-refined Lagrange--Chebyshev approximation, resonances can be calculated efficiently also further into the negative half-space $\{s\in\C\mid \Rea s<0\}$ as well as for Schottky surfaces with several generators and with small funnel widths. This method is not restricted to Schottky surfaces with additional symmetries or any other specific properties.
\end{obsnn}

The basic idea of discretizing operators by expansions in orthogonal
polynomials and approximating Fredholm determinants by matrix
determinants already has a rather long history and in principle dates back to Ritz,
Galerkin~\cite{GanderWanner12} and Nystr{\"o}m~\cite{Nystroem30}. The
review of Bornemann~\cite{Bornemann10} gives a nice overview on the
topic and discusses examples from random matrix theory, among others.
The approach was also applied successfully to integrable quantum
systems~\cite{DugaveEtal14}, where correlation functions are
calculated with the help of quantum transfer matrices. In the realm of
transfer operators it was used, e.g., in~\cite{Fraczek_thesis, Fraczek_book}, mostly applying a monomial basis, leading to new theoretical results~\cite{BFM}. Also computation schemes for transfer operators focusing on Chebyshev polynomials exist already, albeit without discussing their Fredholm determinants, e.g.,~\cite{Nielsen_thesis, Wormell, BanSlip}. The main new aspect
of our approach is the domain-refinement, which leads to a particularly
efficient basis for the discretized transfer operator. In this article, we focus on presenting the numerical method. We will discuss its validity, error estimates and convergence rates in a forthcoming article.

\section{Schottky surfaces, resonances, classical transfer operators}

\subsection{Schottky surfaces}

Schottky surfaces are those hyperbolic surfaces for which the fundamental group is a (Fuchsian) Schottky group. Equivalently, these are the geometrically finite, conformally compact, infinite-area hyperbolic surfaces without orbifold singularities or, also equivalently, the hyperbolic surfaces for which the fundamental group is geometrically finite, convex cocompact, non-cocompact and torsion-free. Every Schottky surface can be obtained by a certain geometric construction algorithm~\cite{Button}, which we will briefly recall in what follows. This algorithm constructs a fundamental domain for every Schottky surface; the side-pairing elements of the fundamental domain and their relations provide a presentation of the fundamental group of the Schottky surface. All hyperbolic surfaces arising from this construction are indeed Schottky.

\subsubsection{Preliminaries for the construction}
As model for the hyperbolic plane we use throughout the upper half-plane
\[
 \h \coloneqq \{ z\in\C \setmid \Ima z > 0\}\,,\qquad ds^2_z \coloneqq \frac{dz\,d\overline z}{(\Ima z)^2}\,,
\]
and we identify its group $\Isom^+(\h)$ of orientation-preserving Riemannian isometries with the Lie group $\PSL_2(\R) = \SL_2(\R)/\{\pm\one\}$. Here, we denote with~$\one$ the matrix $\textmat{1}{0}{0}{1} \in \SL_2(\R)$, and we denote an element in $\PSL_2(\R)$ by 
\[
 \bmat{a}{b}{c}{d} 
\]
if it is represented by the matrix~$\textmat{a}{b}{c}{d}\in\SL_2(\R)$. With all identifications in place, the action of~$\PSL_2(\R)$ on~$\h$ is given by fractional linear transformations, thus
\[
 g.z = \frac{az+b}{cz+d}
\]
for $g=\textbmat{a}{b}{c}{d}\in\PSL_2(\R)$, $z\in\h$.

The hyperbolic plane~$\h$ embeds canonically into the Riemann sphere~$\wh\C = \C\cup\{\infty\}$. In this embedding, the topological boundary~$\partial\h = \R\cup\{\infty\}$ of~$\h$ coincides with its geodesic boundary. The action of~$\PSL_2(\R)$ on~$\h$ extends holomorphically to all of~$\wh\C$. The extended action is given by  
\[
 g.z = 
 \begin{cases}
  \infty & \text{if $z=\infty, c=0$, and if $z\not=\infty, cz+d=0$,}
  \\
  \frac{a}{c} & \text{if $z=\infty, c\not=0$,}
  \\
  \frac{az+b}{cz+d} & \text{otherwise}
 \end{cases}
\]
for $g=\textbmat{a}{b}{c}{d}\in\PSL_2(\R)$ and $z\in\wh\C$. 

\subsubsection{Geometric construction of Schottky surfaces}

\begin{figure}
  {\centering \includegraphics[height=4cm]{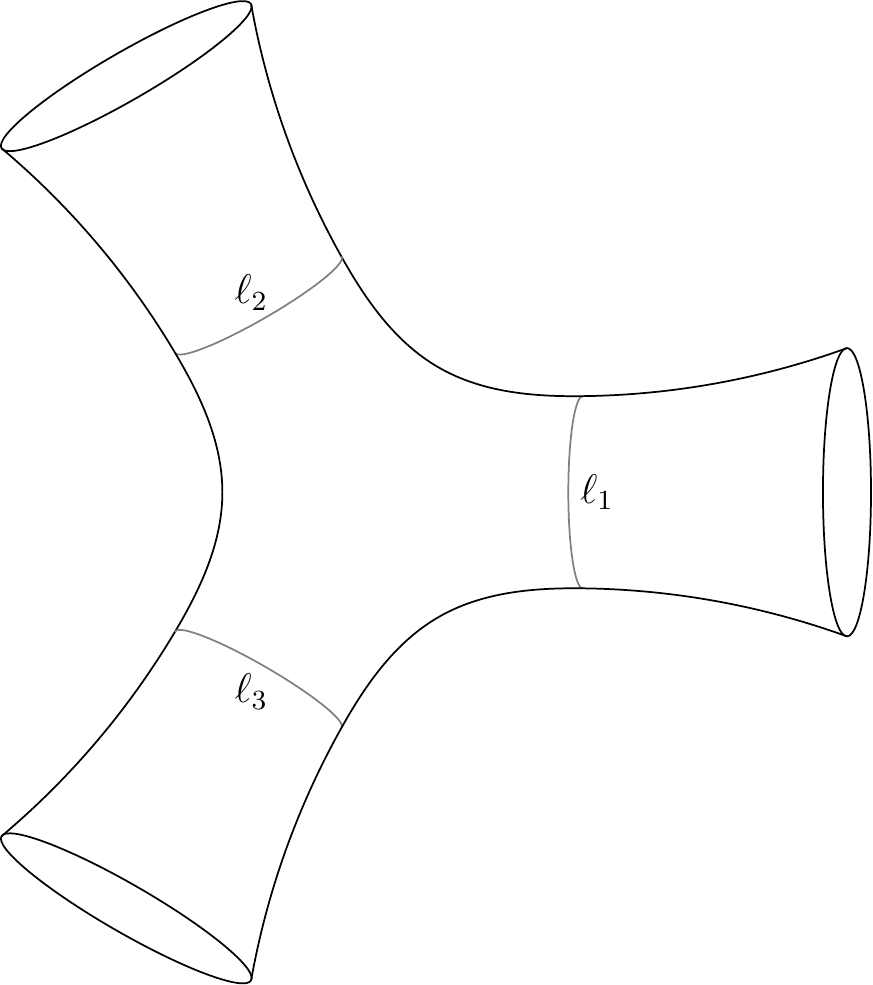}\hfill
    \includegraphics[height=4cm]{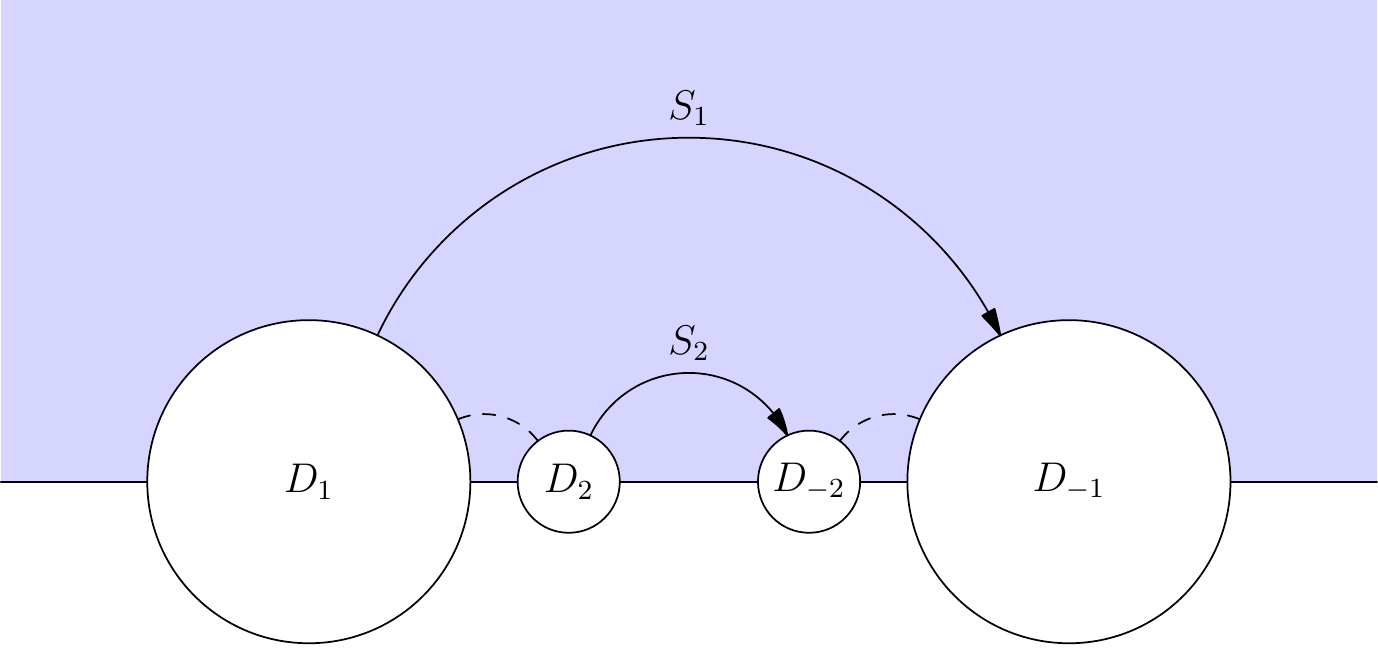} }
  \caption{Construction of a three-funnel
    surface.}\label{fig:c3fun}
\end{figure}

\begin{figure}
  {\centering \includegraphics[width=0.3\linewidth]{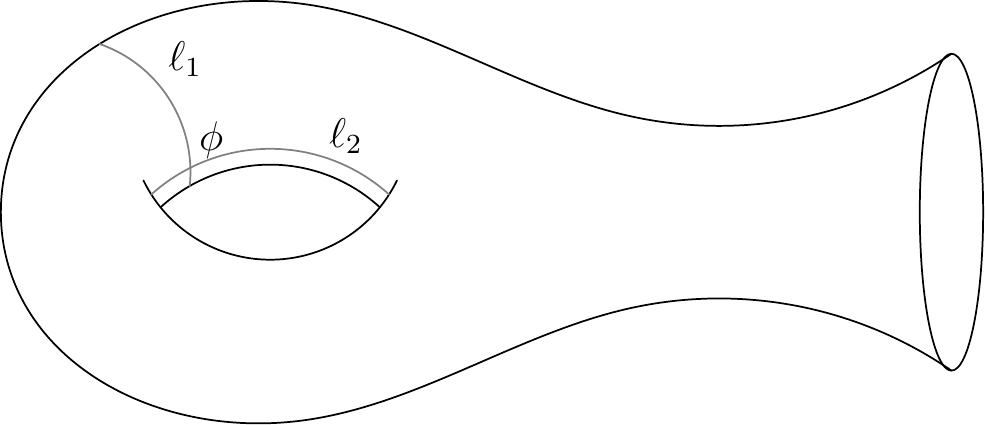}\hfill
    \includegraphics[width=0.6\linewidth]{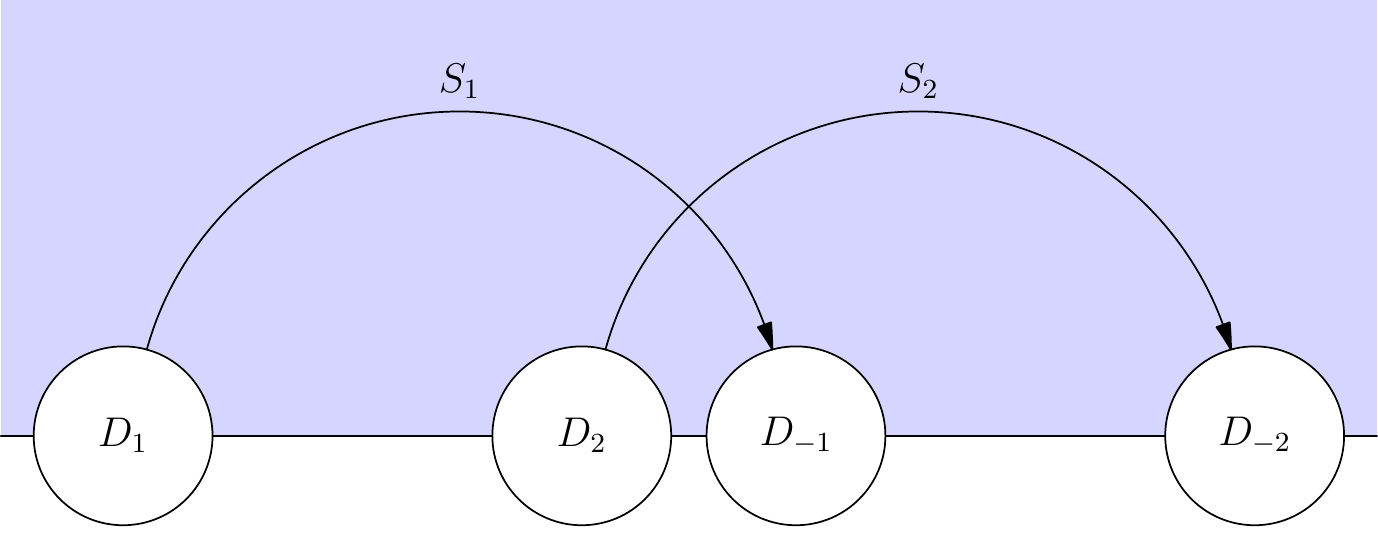} }
  \caption{Construction of a funneled torus surface.}\label{fig:funtorus}
\end{figure}

In order to construct a Schottky surface~$X$ of Euler characteristic $\chi(X)=1-q$ (note that $q\in\N$), we choose $2q$ open Euclidean disks in~$\C$ with centers in~$\R$ and pairwise disjoint closures. We further fix a pairing of these disks, indicated by indices in $I_G\coloneqq \{\pm 1,\ldots, \pm q\}$ with opposite signs, say 
\[
 D_1,D_{-1},\ldots, D_q,D_{-q}\,.
\]
We emphasize that the numeration of these disks implied by their indices is not related to their relative positions in~$\C$. For each $j\in\{1,\ldots, q\}$ we choose an element~$S_j$ in $\PSL_2(\R)$ that maps the exterior of the disk~$D_j$ to the interior of the disk~$D_{-j}$. For examples see Figures~\ref{fig:c3fun} and~\ref{fig:funtorus}. The subgroup
\[
\Gamma = \left\langle S_1,\ldots, S_q \right\rangle
\]
of~$\PSL_2(\R)$ generated by $S_1,\ldots, S_q$ is a \emph{Schottky group}, and the orbit space~$X = \Gamma\backslash\h$ is a \emph{Schottky surface} with Euler characteristic $1-q$. The subset 
\[
 \mc F \coloneqq \bigcap_{k\in I_G} \big( \h\smallsetminus D_k \big)
\]
of~$\h$ is a closed fundamental domain for~$X$ with side-pairings given by the elements~$S_1,\ldots, S_q$. Note that the set of relations among~$S_1,\ldots, S_q$ is empty.

For further use we set 
\[
 S_{-j}\coloneqq S_j^{-1}
\]
for $j\in\{1,\ldots, q\}$, and we call the tuple 
\[
 \left( q, \big(D_k\big)_{k\in I_G}, \big(S_k\big)_{k\in I_G}  \right)
\]
\emph{Schottky data} for~$X$. We remark that different Schottky data may give rise to the same Schottky surface, and indeed any given Schottky surface admits  uncountably many choices of Schottky data. We also remark that the notion of Schottky data in~\cite{APW} is slightly more general; the additional freedom allowed there for the choices of the disk family will not be needed for our purposes.

\subsection{Resonances}

The \emph{resonances} of a Schottky surface~$X$ are the poles of the meromorphic continuation of the resolvent of its Laplacian. To be more precise, we recall that the Laplace operator
\[
 \Delta_\h \coloneqq -y^2 \big(\partial_x^2 + \partial_y^2\big)
\]
on~$\h$ (here, $z=x+iy \in\h,\ x,y\in\R$) induces the Laplace operator 
\[
 \Delta_X \colon L^2(X) \to L^2(X)
\]
on~$X$. The resolvent 
\begin{equation}\label{eq:resolventfirst}
 R_X(s) \coloneqq \big( \Delta_X - s(1-s)\big)^{-1}\colon L^2(X) \to L^2(X)
\end{equation}
of~$\Delta_X$ is defined for all $s\in\C$ with $\Rea s > \tfrac12$ for which $s(1-s)$ is not an $L^2$-eigenvalue of~$\Delta_X$. As shown in~\cite{Mazzeo_Melrose,GZ_upper_bounds}, the restriction of the family of maps in~\eqref{eq:resolventfirst} to $C_c^\infty(X)\to C^\infty(X)$ extends to a meromorphic family 
\[
 R_X(s)\colon L^2_{\mathrm{comp}}(X) \to H^2_{\mathrm{loc}}(X)
\]
on all of~$\C$. The poles of~$R_X$ are the resonances of~$X$. We let $R(X)$ denote the 
multiset of resonances of~$X$ (thus, resonances are included with multiplicities, i.e., the rank of the residue of $R_X$ at the resonance).

\subsection{Selberg zeta function}

For any Schottky surface~$X$ we denote by $L(X)$ its multiset of lengths of primitive periodic geodesics (thus, the prime geodesic length spectrum including multiplicities). Further we denote by~$\Lambda(X)$ the limit set of~$X$. We recall that $\Lambda(X)$ is the set of limit points (i.e., accumulation points) in the Riemann sphere~$\wh\C$ of the orbit $\Gamma.z$ for some (and hence any) $z\in\h$ and any realization~$\Gamma$ of the fundamental group of~$X$ in~$\PSL_2(\R)$. As is well-known, $\Lambda(X)$ is a subset of $\partial\h = \R\cup\{\infty\}$. We let $\delta=\delta(X)$ denote the Hausdorff dimension of~$\Lambda(X)$, which for Schottky surfaces can be any value in~$[0,1)$.

The \emph{Selberg zeta function}~$Z_X$ of~$X$ is given by the infinite product
\begin{equation}\label{eq:szf}
 Z_X(s) = \prod_{\ell\in L(X)} \prod_{k=0}^\infty \left( 1 - e^{-(s+k)\ell} \right)
\end{equation}
for $s\in\C$ with $\Rea s > \delta$ (where this infinite product converges compactly), and beyond this range by the holomorphic continuation of~\eqref{eq:szf} to all of~$\C$. The multiset of zeros of~$Z_X$ consists of the multiset~$R(X)$ of resonances of~$X$ and the \emph{topological zeros}. The latter are well-understood: they are located at~$-\N_0$ with well-known multiplicities. For proofs and more details we refer to the original articles~\cite{Patterson_Perry} as well as~\cite{Borthwick_Judge_Perry05}. (We remark that the relation between resonances and topological zeros of Selberg zeta functions is reminiscent of the relation between trivial and nontrivial zeros of the Riemann zeta function.)

\subsection{Transfer operators}\label{sec:def_TO}

The choice of Schottky data 
\[
 \mc S\coloneqq \left( q, \big(D_k\big)_{k\in I_G}, \big(S_k\big)_{k\in I_G} \right)
\]
for a Schottky surface~$X$ allows us to construct a Ruelle-type transfer operator family whose underlying discretization of the geodesic flow on~$X$ arises from a Koebe--Morse coding based on the fundamental domain induced by~$\mc S$. We refer to \cite{BorthwickBook2nd} for  details and proofs, and present here only the resulting transfer operators and their properties. 

To simplify notation, for $s\in\C$, any subset~$U\subseteq\R$, any function~$f\colon U\to\C$ and any element~$g=\textbmat{a}{b}{c}{d}\in\nobreak\PSL_2(\R)$ we set
\[
 \tau_s(g^{-1})f(x) \sceq \big( g'(x) \big)^s f(g.x) = |cx+d|^{2s}f\left(\frac{ax+b}{cx+d}\right)\qquad \text{($x\in U$)}
\]
whenever this is well-defined. We will also use holomorphic extensions of this definition to functions defined on open subsets of~$\C$ (e.g., the disks~$D_k$, $k\in I_G$). We omit the discussion of the possible choices of holomorphic extensions, in particular of the choices for the logarithm, which heavily depend on the Schottky data and on the required combinations of functions~$f$ and group elements~$g$ in the Schottky group. None of the stated results depend on these choices. We continue to denote the holomorphic extensions by~$\tau_s$ or, more precisely, by~$\tau_s(g^{-1})f$.

The \emph{transfer operator} $\TO_s$ with parameter $s\in\C$ associated to $\mc S$ is defined formally by 
\begin{equation}\label{eq:def_TO1}
 \TO_s \coloneqq \sum_{k\in I_G} 1_{D_k} \cdot \sum_{\substack{j\in I_G\\ j\not=-k}} \tau_s(S_j)\,,
\end{equation}
where $1_{D_k}$ denotes the characteristic function of~$D_k$. 

The choice of function space on which to consider the transfer operator~$\TO_s$ is guided by the required applications and may be a subtle task. For Selberg zeta functions there are several good choices of which we present a few in what follows. In Section~\ref{sec:domainrefined} below we will present some more spaces crucial for our investigations. Throughout, we will denote all transfer operators by~$\TO_s$, independently of the function space on which they are considered to act. 

One common choice, widely used in the study of resonances of Schottky surfaces via transfer operators \cite{GLZ, BorthwickBook2nd}, is the \emph{Hilbert--Bergman space} given by 
\begin{equation}\label{eq:HBspace}
 \mc H \coloneqq \bigoplus_{k\in I_G} \mc H(D_k)\,,
\end{equation}
where 
\[
 \mc H(D_k) \coloneqq \left\{ \text{$f\colon D_k\to\C$ holomorphic} \setmid \int_{D_k} |f(z)|^2\,d\lambda(z) < \infty \right\}
\]
is the Hilbert space of holomorphic square-integrable functions on~$D_k$, endowed with the standard $L^2$-inner product. Here, $d\lambda$ denotes the Lebesgue measure on~$\C$.

Another common choice, already used by Ruelle~\cite{Ruelle} and Mayer~\cite{Mayer_gauss,Mayer_thermo} in their seminal investigations of Selberg zeta functions via transfer operators, is the \emph{disk algebra}~$\mc B$ associated to~$\mc S$. This is the Banach space
\begin{equation}\label{eq:DAspace}
 \mc B = \bigoplus_{k\in I_G} A_\infty(D_k)\,,
\end{equation}
where 
\[
 A_\infty(D_k) \coloneqq \left\{ \text{$f\colon D_k\to\C$ holomorphic} \setmid \text{$f$ extends continuously to $\overline D_k$}\right\}
\]
is the Banach space of holomorphic functions on~$D_k$ that are restrictions of continuous functions on the closure~$\overline D_k$ of~$D_k$, endowed with the supremum norm.

On the spaces in~\eqref{eq:HBspace} or in~\eqref{eq:DAspace} (as well as many other spaces),  the transfer operator $\TO_s$ is a well-defined self-map, a nuclear operator of order~$0$ (thus, trace class on~$\mc H$) and its Fredholm determinant equals the Selberg zeta function of~$X$, thus
\[
 Z_X(s) = \det(1-\TO_s)
\]
for all $s\in\C$ (see, e.g., \cite{BorthwickBook2nd} and~\cite{FP_szf}). Therefore, with either of these spaces as domain, the transfer operator~$\TO_s$ has an eigenfunction with eigenvalue~$1$ if and only if $s$ is a zero of $Z_X$. See, e.g.,~\cite{GGK}.

\section{Domain-refined transfer operators}\label{sec:domainrefined}

Throughout let $X$ be a Schottky surface, 
\[
 \mc S = \left( q, \big(D_k\big)_{k\in I_G}, \big(S_k\big)_{k\in I_G}\right)
\]
a choice of Schottky data for~$X$, and $(\TO_s)_{s\in\C}$ the associated transfer operator family. In Section~\ref{sec:def_TO} we have exhibited two classes of function spaces that are by now ``classical'' domains for the transfer operator~$\TO_s$. Both spaces consist of functions on the disks~$D_k$, $k\in I_G$. In this section we will construct spaces of functions with smaller domains, a step that will be crucial for our numerical method. 

The construction of smaller domains is guided by the following intuition: in order to calculate resonances of~$X$ we want to use the Selberg zeta function~$Z_X$ and its 
representation $\det(1-\TO_s)$ as a Fredholm determinant by a transfer operator family on a suitable function space, say~$H$. As can be seen from the infinite product in~\eqref{eq:szf}, the function~$Z_X$ is determined by the \emph{periodic} geodesics on~$X$. Given any such periodic geodesic~$\wh\gamma$ on~$X$ and any geodesic~$\gamma$ on~$\h$ that represents~$\wh\gamma$, the two (time-)endpoints $\gamma(+\infty)$ and $\gamma(-\infty)$ of~$\gamma$ are in the limit set~$\Lambda(X)$ of~$X$. Intuitively, the function space~$H$ should consist of functions whose domains enclose~$\Lambda(X)$ and stay ``near''~$\Lambda(X)$, 
the nearer to~$\Lambda(X)$ the better the numerics. For the choice of~$H$ we need to guarantee that the transfer operators~$\TO_s$, $s\in\C$, define endomorphisms of~$H$ and that the relation $Z_X=\det(1-\TO_s)$ holds. 
To that end we will start with the disks~$D_k$, $k\in I_G$, as domains for the functions and then use the iterated action of the elements~$S_k$, $k\in I_G$, to obtain a sequence of refined domains 
converging to~$\Lambda(X)$. 

We first restrict the discussion of the domains to~$\R$ and present the shrinking algorithm on the level of real intervals. We then enlarge the shrunken intervals to 
domains in~$\C$. The construction thus allows for an additional degree of freedom which we will exploit for the construction of suitable function spaces. Figure~\ref{fig:refinement} gives an illustration of the iterated refinement procedure that we describe in what follows. 

\begin{figure}[H]
  {\centering \includegraphics[width=.6\textwidth]{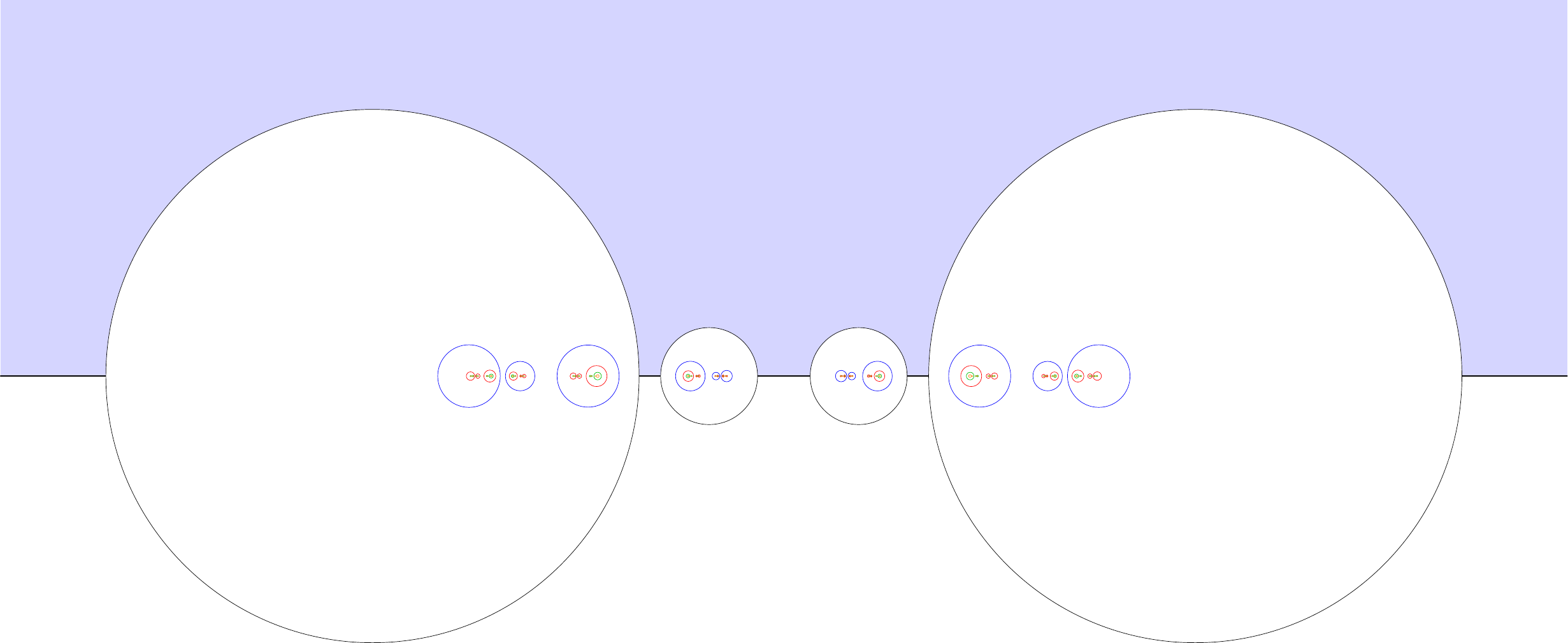}}
  \caption{Recursive subdivision (for simplicity drawn with disks).}\label{fig:refinement}
\end{figure}

For $k\in I_G$ we let 
\[
 I_k \coloneqq D_k\cap \R
\]
denote the interval in~$\R$ that is enclosed by the disk~$D_k$, and we restrict the transfer operator in~\eqref{eq:def_TO1} to these intervals. Thus, again initially only formally,
\begin{equation}\label{eq:TO_basic}
 \TO_s = \sum_{k\in I_G} 1_{I_k} \cdot \sum_{j\in I_G\smallsetminus\{-k\}} \tau_s(S_j)\,.
\end{equation}
For each $n\in\N_0$ we let $\mc W_n$ denote the set of tuples $(w_1,\ldots, w_n,\ell) \in I_G^{n+1}$ such that $w_k\not=-w_{k-1}$ for all $k\in\{2,\ldots,n\}$ and $\ell\not=w_n$. For each $w=(w_1,\ldots, w_n,\ell)\in\mc W_n$ we set
\[
 I_w \coloneqq S_{w_1}S_{w_2}\cdots S_{w_n}.I_\ell
\]
and we let 
\[
 \Fct(I_w;\C) \coloneqq\{ f_w\colon I_w\to\C\}
\]
denote the space of complex-valued functions on the interval~$I_w$. Further we let 
\[
 \mc I_n \coloneqq \{ I_w \setmid w\in\mc W_n\}
\]
denote the set of \emph{intervals of refinement level~$n$} and we set 
\begin{equation}\label{eq:def_Fct}
 \Fct(\mc I_n;\C) \coloneqq \bigoplus_{w\in\mc W_n} \Fct(I_w;\C)\,.
\end{equation}
The following proposition shows that we may consider the transfer operator~$\TO_s$ as an operator on~$\Fct(\mc I_n;\C)$. We omit the straightforward proof.

\begin{prop}\label{prop:TOmatrix}
For each $n\in\N$ and $s\in\C$ the transfer operator $\TO_s$ in~\eqref{eq:TO_basic} defines a self-map on $\Fct(\mc I_n;\C)$. 
\end{prop}

Let $n\in\N$, $s\in\C$ and consider the transfer operator $\TO_s$ as an operator on $\Fct(\mc I_n;\C)$. The definition of~$\Fct(\mc I_n;\C)$ in~\eqref{eq:def_Fct} as a direct sum of the function spaces~$\Fct(I_w;\C)$, $w\in\mc W_n$, yields an identification of~$\TO_s$ with a matrix
\begin{equation}\label{eq:reprmatrix}
 \left( \TO_{s,v,w}\right)_{v,w\in\mc W_n}
\end{equation}
where the \emph{matrix coefficient} $\TO_{s,v,w}$ for $v,w\in\mc W_n$ is the unique operator 
\[
 \TO_{s,v,w}\colon \Fct(I_w;\C) \to \Fct(I_v;\C)
\]
such that for all $f,\tilde f\in\Fct(\mc I_n;\C)=\bigoplus\Fct(I_w;\C)$ with 
\[
 f = \bigoplus_{w\in\mc W_n}f_w, \qquad \tilde f = \bigoplus_{w\in\mc W_n}\tilde f_w
\]
and $\TO_s f = \tilde f$ we have
\[
 \tilde f_v = \sum_{w\in\mc W_n} \TO_{s,v,w}f_w\,.
\]
A straightforward calculation allows us to deduce the following explicit formulas for the matrix coefficients.

\begin{prop}\label{prop:TOcoeff} Let $s\in\C$.
\begin{enumerate}[label=(\roman*),font=\normalfont]
\item For $n=0$ and $v,w\in\mc W_0 = I_G$ the matrix coefficient~$\TO_{s,v,w}$ is
\[
 \TO_{s,v,w} =
 \begin{cases}
\tau_s(S_w) & \text{if $w\not=-v$}
\\
0 & \text{otherwise.}
 \end{cases}
\]
\item For $n\geq 1$ and $v=(v_1,\ldots, v_n,\ell_v), w=(w_1,\ldots, w_n, \ell_w)\in\mc W_n$ the matrix coefficient~$\TO_{s,v,w}$ is 
\[
 \TO_{s,v,w} = 
 \begin{cases}
  \tau_s(S_{-w_1}) & \text{if $w=(w_1,v_1,\ldots,v_{n-1},-v_n)$}
  \\
  0 & \text{otherwise.}
 \end{cases}
\]
\end{enumerate}
\end{prop}

With increasing refinement level~$n$, the matrix in~\eqref{eq:reprmatrix} representing the transfer operator $\TO_s\colon\Fct(\mc I_n;\C) \to \Fct(\mc I_n;\C)$ becomes sparse with exponentially increasing sparsity.

\begin{example}
We consider a Schottky surface with two generators, say $S_1$ and $S_2$. For refinement level~$0$ the transfer operator~$\TO_s$ is identified with  the matrix
\begin{equation}\label{eq:refine0}
  \begin{pmatrix}
  \tau_s(S_{-2}) & \tau_s(S_{-1}) & \tau_s(S_1) & 0
  \\
  \tau_s(S_{-2}) & \tau_s(S_{-1}) & 0 & \tau_s(S_2)
  \\
  \tau_s(S_{-2}) & 0 & \tau_s(S_1) & \tau_s(S_2)
  \\
  0 & \tau_s(S_{-1}) & \tau_s(S_1) & \tau_s(S_2)
 \end{pmatrix}\,,
\end{equation}
acting on function vectors of the form 
\[
 \begin{pmatrix}
  f_{-2} \colon I_{-2} \to \C
  \\
  f_{-1} \colon I_{-1} \to \C
  \\
  f_1 \colon I_1 \to \C
  \\
  f_2 \colon I_2 \to \C
 \end{pmatrix}\,.
\]
For refinement level~$1$, the transfer operator~$\TO_s$ is identified with the matrix
\begin{small}\setlength{\arraycolsep}{2pt}
\begin{align}\label{eq:refine1}
\begin{pmatrix}[ccc|ccc|ccc|ccc]
   0 & S_{-2} & 0 & S_{-1} & 0 & 0 & S_1 & 0 & 0 & 0 & 0 & 0
   \\
   0 & S_{-2} & 0 & S_{-1} & 0 & 0 & S_1 & 0 & 0 & 0 & 0 & 0
   \\
   0 & S_{-2} & 0 & S_{-1} & 0 & 0 & S_1 & 0 & 0 & 0 & 0 & 0
   \\\hline
   0 & 0 & S_{-2} & 0 & S_{-1} & 0 & 0 & 0 & 0 & 0 & 0 & S_2
   \\
   0 & 0 & S_{-2} & 0 & S_{-1} & 0 & 0 & 0 & 0 & 0 & 0 & S_2
   \\
   0 & 0 & S_{-2} & 0 & S_{-1} & 0 & 0 & 0 & 0 & 0 & 0 & S_2
   \\\hline
   S_{-2} & 0 & 0 & 0 & 0 & 0 & 0 & S_1 & 0 & S_2 & 0 & 0
   \\
   S_{-2} & 0 & 0 & 0 & 0 & 0 & 0 & S_1 & 0 & S_2 & 0 & 0
   \\
   S_{-2} & 0 & 0 & 0 & 0 & 0 & 0 & S_1 & 0 & S_2 & 0 & 0
   \\\hline
   0 & 0 & 0 & 0 & 0 & S_{-1} & 0 & 0 & S_1 & 0 & S_2 & 0
   \\
   0 & 0 & 0 & 0 & 0 & S_{-1} & 0 & 0 & S_1 & 0 & S_2 & 0
   \\
   0 & 0 & 0 & 0 & 0 & S_{-1} & 0 & 0 & S_1 & 0 & S_2 & 0
\end{pmatrix}\,,
\end{align}
\end{small}
acting on function vectors of the form 
\[
(f_{2,1}, f_{2,-2}, f_{2,-1}, f_{1,-2}, f_{1,-1}, f_{1,2}, f_{-1,-2}, f_{-1,1}, f_{-1,2}, f_{-2,1}, f_{-2,2}, f_{-2,-1})^\top 
\]
with 
\[
 f_{a,b}\colon S_{a}.I_b\to \C\,.
\]
In~\eqref{eq:refine1}, each entry of the form~$S_j$ indicates the operator~$\tau_s(S_j)$. The matrix representation of~$\TO_s$ in~\eqref{eq:refine1} shows that each entry in~\eqref{eq:refine0} is changed for a \mbox{$3\ttimes 3$-}matrix containing mostly zeros. 
\end{example}

Let $n\in\N$, $s\in\C$. We now thicken the intervals in~$\mc I_n$ to a suitable family of domains in~$\C$ so that we can find a good function space~$H$ on which the transfer operators~$\TO_s$ ($s\in\C$) act and satisfy the relation $Z_X(s) = \det(1-\TO_s)$. 

To that end we call, for any $n\in\N_0$, a collection~$\mc E_n \coloneqq \{ E_w \setmid w\in\mc W_n\}$ of subsets of~$\C$ a \emph{family of admissible neighborhoods for~$\mc I_n$} if for each $w\in \mc W_n$ the set~$E_w$ is an open, bounded and convex subset of~$\C$ with $I_w\subseteq E_w$, the elements of~$\mc E_n$ are pairwise disjoint, and 
\begin{itemize}
 \item in case that $n=0$, we have 
 \[
 S_w^{-1}.\overline E_v \subseteq E_w
 \]
 whenever $w\not=-v$, and
\item if $n\geq 1$, we have
\[
 S_{-w_1}^{-1}.\overline E_v \subseteq E_w
\]
whenever $w=(w_1,v_1,\ldots, v_{n-1},-v_n)$.
\end{itemize}
For any open subset~$D$ of~$\C$ we let~$H(D)$ denote an ``admissible'' Banach or Hilbert space of functions~$D\to\C$. Admissible spaces include, e.g., the Hilbert--Bergman spaces and the disk algebras from Section~\ref{sec:def_TO} as well as Hilbert--Hardy spaces. For a wider-ranging notion of admissibility see~\cite{Bandtlow_Jenkinson}. In the forthcoming article in which we will discuss the validity, error estimates and convergence rates of the method of domain-refined 
Lagrange--Chebyshev approximation we will use Hilbert--Hardy spaces over ellipses. For any family~$\mc E_n$ of admissible neighborhoods of~$\mc I_n$ we set 
\[
 H(\mc E_n) \coloneqq \bigoplus_{w\in \mc W_n} H(E_w)\,.
\]

\clearpage 

\begin{thm}\label{thm:SZF_refineddomain} Let $n\in\N_0$.
\begin{enumerate}[label=(\roman*),font=\normalfont]
\item\label{rdi} There exists an admissible family of neighborhoods~$\mc E_n$ of~$\mc I_n$. 
\item\label{rdii} For every admissible family~$\mc E_n$ of neighborhoods of~$\mc I_n$, the transfer operator~$\TO_s$ extends to an operator~$\TO_s$ on~$H(\mc E_n)$. 
\item\label{rdiii} The operator~$\TO_s$ on~$H(\mc E_n)$ is trace class (or nuclear of order~$0$) and its Fredholm determinant represents the Selberg zeta function. Thus, for all~$s\in\nobreak\C$ we have
\[
 Z(s) = \det(1-\TO_s)\,.
\]
\end{enumerate}
\end{thm}

\begin{proof}
To establish~\ref{rdi} we let, for any $w\in\mc W_n$, the set~$E_w$ denote the open disk in~$\C$ with center in~$\R$ that satisfies $E_w\cap\R = I_w$. A straightforward calculation, taking advantage of the fact that Moebius transformations preserve disks, shows that $\{E_w\}_{w\in\mc W_n}$ is a family of admissible neighborhoods of~$\mc I_n$. Since for $n=0$, $\mc E_0$ being a family of disks and $H(\mc E_0)$ being chosen as Hilbert--Bergman spaces, the validity of~\ref{rdii} and~\ref{rdiii} is known (both statements follow directly from~\cite{GLZ}), claims~\ref{rdii} and~\ref{rdiii} for arbitrary $n\in\N$, arbitrary families~$\mc E_n$ of admissible neighborhoods of~$\mc I_n$ and arbitrary admissible function spaces~$H(\mc E_n)$ now follow either directly from~\cite{FP_szf} or with obvious adaptions. (Establishing the existence of a family of ellipses satisfying~\ref{rdi} is a non-trivial matter and will be discussed in detail in the forthcoming follow-up article.)
\end{proof}

\section{Lagrange--Chebyshev approximation}\label{sec:LC}

Let $X$ be a Schottky surface, 
\[
 \mc S = \left( q, \big(D_k\big)_{k\in I_G}, \big(S_k\big)_{k\in I_G}\right)
\]
a choice of Schottky data for~$X$, $n\in\N_0$ and $\mc E_n$ a family of admissible neighborhoods of~$\mc I_n$ such that for each $w\in\mc W_n$, the set~$E_w$ is an ellipse with foci on the boundary points of the interval~$I_w$, let $H(\mc E_n)$ be the Hilbert--Hardy space, and let~$\TO_s$ denote the transfer operator family associated to~$\mc S$, considered as an operator on~$H(\mc E_n)$. We recall that the Selberg zeta function $Z_X$ equals the Fredholm determinant of~$(\TO_s)_{s\in\C}$, thus, for all $s\in\C$ we have $Z_X(s) = \det(1-\TO_s)$.

The Lagrange--Chebyshev approach for the numerical calculation of the
Selberg zeta function consists of three basic steps:
\begin{enumerate}[label=(\roman*),font=\normalfont]
\item \label{LC-1} expand functions in $H(\mc E_n)$ in Chebyshev series, and use Gauss--Chebyshev quadrature to approximate the integrals in the coefficients,
\item \label{LC-2} describe the action of $1-\TO_s$ on the Chebyshev series, thus on the Chebyshev polynomials, and approximate these actions in the sense of Lagrange, resulting in an approximation by a finite-dimensional matrix, 
\item \label{LC-3} approximate $Z_X(s)$ by the determinant of this matrix.
\end{enumerate}
In steps~\ref{LC-1} and~\ref{LC-2} the ideas of Lagrange and Chebyshev
come into play, as we use a Chebyshev expansion to construct Lagrange
interpolating polynomials. Steps~\ref{LC-2} and~\ref{LC-3} are
reminiscent of Nystr\"om's discretization scheme for Fredholm integral
equations~\cite{Nystroem30}, which leads to a highly efficient
numerical method for the evaluation of Fredholm determinants of
integral operators~\cite{Bornemann10}.

\subsection{Interpolating polynomials}

A function~$f$ defined real-analytically in the neighborhood of the real interval~$[-1,1]$ can
be expanded in terms of Chebyshev polynomials of the first kind,
\begin{equation}
  f(x) = \mu_0 + 2 \sum_{k=1}^{\infty} \mu_k T_k(x)\,, \qquad (x\in [-1,1])
\end{equation}
with polynomials
\begin{equation}
  T_k(x) \coloneqq \cos(k \arccos(x))
\end{equation}
and expansion coefficients
\begin{equation}
  \mu_k = \int_{-1}^{1} \frac{f(x) T_k(x)}{\pi \sqrt{1-x^2}}\, dx\,.
\end{equation}
These integrals can be approximated using Gauss--Chebyshev
quadrature,
\begin{equation}
  \int_{-1}^{1} \frac{g(x)}{\pi \sqrt{1-x^2}}\, dx \approx \sum_{j=1}^{N} w_j\ g(x_j)\,,
\end{equation}
with points and weights
\begin{equation}
  x_j = \cos\left(\frac{2j-1}{2N}\pi\right),\qquad w_j = \frac{1}{N}\,,\qquad j=1,\ldots,N\,.
\end{equation}
For the expansion coefficients this leads to
\begin{equation}\label{muapprox}
  \mu_k \approx \sum_{j=1}^{N} w_j\ T_k(x_j)\ f(x_j) \,.
\end{equation}
Hence, we can define a Chebyshev kernel $K_M(x, y)$ and approximate
the function $f$ by its values at the quadrature points $x_j$,
\begin{equation}\label{fxapprox}
  \begin{aligned}
    K_M(x, y) & = \frac{1}{M} \left[T_0(x) T_0(y) + 2 \sum_{k=1}^{M-1} T_k(x) T_k(y)\right]\,,\\
    f(x) & \approx \frac{M}{N} \sum_{j=1}^{N} K_M(x, x_j)\ f(x_j)\,.
  \end{aligned}
\end{equation}

Which order $M$ should be used for the kernel $K_M$? The approximation in~\eqref{fxapprox}
is a polynomial of degree $M-1$, which is uniquely defined specifying $N=M$ data points.
In addition, we notice that inserting the definitions of $T_k(x)$, $x_j$ and $w_j$ into~\eqref{muapprox} leads to
\begin{equation}
  \mu_k   = \frac{1}{N} \sum_{j=1}^{N} \cos\left(\frac{\pi}{N} k (j-\tfrac{1}{2})\right) f(x_j)
  \quad\text{with}\quad
  k=0,\dots,N-1\,,
\end{equation}
i.e., a discrete cosine transform of type II~\cite{DCT-1974,DCT-Review}.
Its inverse is the discrete cosine transform of type III, which coincides with the
Chebyshev expansion~\eqref{fxapprox} truncated at order $M=N$ and evaluated at the
points $x=x_j$,
\begin{equation}
  f(x_j) = \mu_0 + 2 \sum_{k=1}^{N-1} \mu_k \cos\left(\frac{\pi}{N} k (j-\tfrac{1}{2})\right)
  \quad\text{with}\quad
  j=1,\dots,N\,.
\end{equation}
Truncation at order $M=N$ is therefore a natural and consistent
choice, which leads to
\begin{equation}
  K_N(x_i, x_j) = \delta_{ij}\,,
\end{equation}
and \eqref{fxapprox} then represents a Lagrange interpolation
polynomial which is exact at the quadrature points $x_j$.
The approximation error in~\eqref{fxapprox} with $M=N$ equals
\begin{equation*}
  \frac{f^{(N)}(\xi_x)}{N!}\Phi_N(x)
\end{equation*}
with $\Phi_N(x) = \prod_{j=1}^{N}(x-x_j)$ and appropriate $\xi_x\in I$.
Compared to other choices of the collocation points, the interpolation
with Chebyshev quadrature points $x_j$ has the advantage to minimize the maximum
magnitude of $\Phi_n(x)$. See, e.g.,~\cite{DahlquistBjoerk,Rivlin74}.

For real-analytic functions $f$ defined in the neighborhood of a generic interval $[a,b]$ we
can rescale linearly,
\begin{equation}
  [a,b]\ni y = c + r x\quad\text{with}\quad x\in[-1,1],\ c = \frac{a+b}{2},\ r=\frac{b-a}{2}\,.
\end{equation}
This leads to the approximations
\begin{equation}
  \begin{aligned}
    f(c + r x) & \approx \sum_{j=1}^{N} K_N(x,x_j)\ f(c + r x_j)
    \qquad\text{or}\\
    f(y)& \approx \sum_{j=1}^{N} K_N[(y-c)/r,(y_j-c)/r]\ f(y_j)\,.
  \end{aligned}
\end{equation}

\subsection{Discretization of the transfer operator}

Each matrix coefficient~$\TO_{s,v,w}$, $v,w\in\mc W_n$, of the transfer operator $\TO_s$ is either the zero operator or an operator~$\tau_s(S_k)$, $k\in I_G$, applied to functions~$f_w$ on a neighborhood (here, an ellipse) of the interval~$I_w$ resulting in functions~$f_v$ on a neighborhood of the interval~$I_v$. See Proposition~\ref{prop:TOcoeff}. Each such matrix coefficient of $\TO_s$ can be discretized
using~\eqref{fxapprox}, i.e., approximated by a finite $N\ttimes N$-matrix. Since all collocation points are contained in the intervals, we may and shall restrict the further discussion to (real-analytic) functions on these intervals, ignoring their extension to complex neighborhoods.

Let $I_w, I_v \in \mc I_n$ be intervals and let $S_k$, $k \in I_G$, be one
of the generators.  Denote the centers and radii of the intervals by
$c_w,c_v$ and $r_w,r_v$, respectively.  The $N$-point approximation
$f_w^{(N)}$ of a function $f_w\in C^\omega(I_w;\C)$ is then given by
\begin{equation}
  \begin{aligned}
    f_w^{(N)}(y) & = \sum_{j=1}^{N} K_N[(y-c_w)/r_w,(y_j-c_w)/r_w]\ f_w(y_j)\,,\quad\text{with}\\
    y_j & = c_w + r_w x_j\,,\quad x_j =
    \cos\left(\tfrac{2j-1}{2N}\pi\right)\,,\quad j=1,\ldots,N\,.
  \end{aligned}
\end{equation}
Applying the matrix coefficient $\TO_{s,v,w}$ to $f_w^{(N)}$ we obtain either zero
or
\begin{equation}
  \begin{aligned}
    \TO_{s,v,w} f_w^{(N)}(y) & = \tau_s(S_k) f_w^{(N)}(y)
    = \big(S_{-k}'(y)\big)^s f_w^{(N)}(S_{-k}.y)\\
    & = \big(S_{-k}'(y)\big)^s \sum_{j=1}^{N}
    K_N[(S_{-k}.y-c_w)/r_w,(y_j-c_w)/r_w]\ f_w(y_j)\,,
  \end{aligned}
\end{equation}
which is a function~$f_v$ on $I_v$ that we can evaluate at the quadrature
points
\begin{equation}
  y_i = c_v + r_v x_i\quad\text{with}\quad
  x_i = \cos\left(\tfrac{2i-1}{2N}\pi\right)\,,\quad i=1,\ldots,N\,.
\end{equation}
The order-$N$ discretization~$\TO_{s,v,w}^{(N)}$ of~$\TO_{s,v,w}$ is the linear operator which maps
function values at quadrature points in $I_w$,
$f_w(c_w + r_w x_j\in I_w)$, to function values at quadrature points in
$I_v$, $f_v(c_v + r_v x_i\in I_v)$. Its $N\ttimes N$-matrix elements are
given by
\begin{equation}\label{def:discrTO}
  \Big(\TO_{s,v,w}^{(N)}\Big)_{i,j}
  = \begin{cases}
    \big(S_{-k}'(c_v + r_v x_i)\big)^s K_N\!\left[\dfrac{S_{-k}.(c_v + r_v x_i) - c_w}{r_w}, x_j\right] & \text{if }S_{-k}.I_v \subseteq I_w\\[2mm]
    0 & \text{otherwise.}
  \end{cases}
\end{equation}
To get an order-$N$ discretization~$\TO_s^{(N)}$ of the full transfer
operator we replace each matrix coefficient~$\TO_{s,v,w}$ by its discretization~$\TO_{s,v,w}^{(N)}$. This results in a big square matrix constructed from many $N\ttimes N$-blocks, most of which are zero. The
full dimension of the vector space the discretization~$\TO_s^{(N)}$ is acting on is~$2q(2q-1)^n N$.

\subsection{Numerical evaluation of matrix determinants}

The transfer operator consists of an $s$-dependent part and a static
part, which only depends on the given intervals and group
generators. When $\TO_s^{(N)}$ is evaluated numerically for different
values of~$s$, only the first part needs to be changed, the second can
be re-used.

To calculate the numerical value of the Selberg zeta function~$Z_X$ at~$s$, thus $Z_X(s) = \det(1 - \TO_s)$, we
construct the matrix $1 - \TO_s^{(N)}$ and calculate its
determinant with some standard numerical library routine. Such routines are
usually based on $LU$-factorization of the matrix, and available code
is highly optimized and fast. For dense matrices (low domain
refinement) we can use \texttt{zgetrf} from LAPACK, for sparse
matrices (high domain refinement) SuiteSparse offers the
\texttt{umfpack} family of functions. Higher level languages such as
Python, Sage or Julia offer simple wrapper functions for matrix
determinants (and for $\log\det[.]$), which select and call one of the
above optimized libraries.

\subsection{Finding zeros of $Z_X$}

\begin{figure}
  \centerline{
    \includegraphics[width=0.49\linewidth]{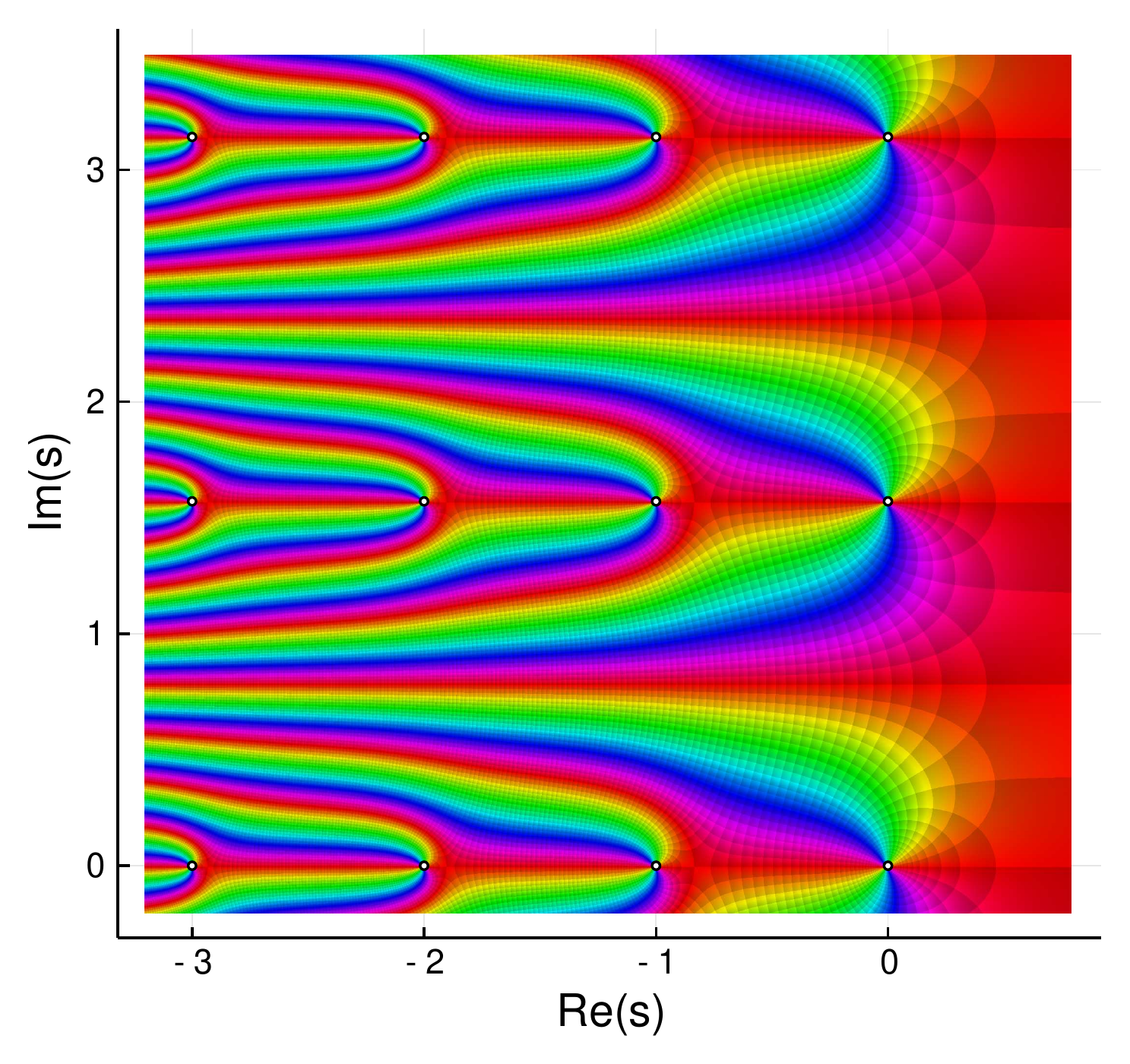}
    \includegraphics[width=0.49\linewidth]{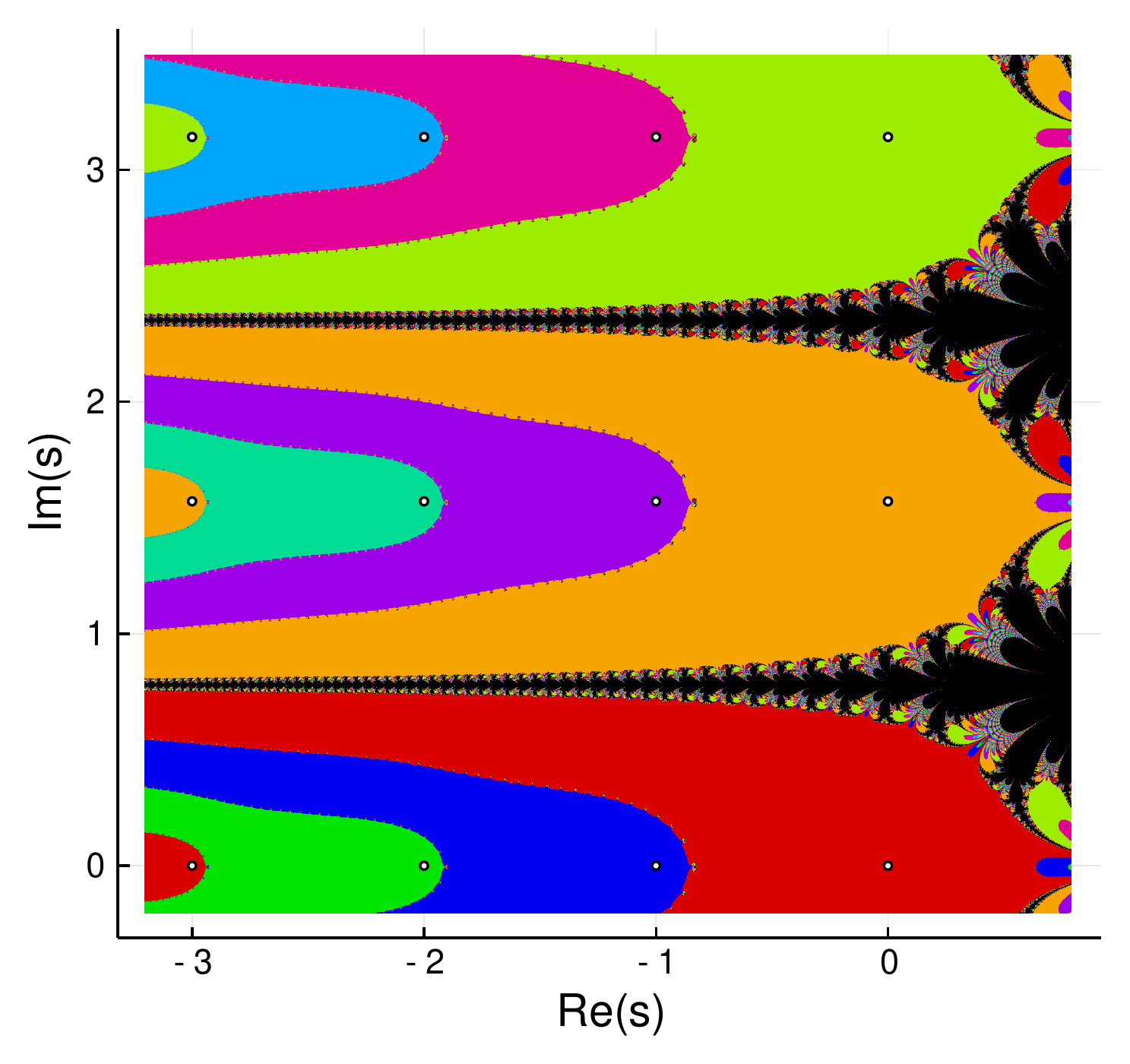}
  }
  \caption{Phase plot of $Z_X(s)$ (left) and Newton fractal (right) for
  the hyperbolic cylinder with funnel width $4$.}\label{fig:newton4cylinder}
\end{figure}

With a stable algorithm for $Z_X(s)$ at hand we are now in the position
to calculate resonances, which are zeros of $Z_X$. A simple and
efficient algorithm for this task is Newton's method, which is based
on the iteration
\begin{equation}
  s\to s - \frac{Z_X(s)}{Z'_X(s)}\,.
\end{equation}
The absolute value of $Z_X(s)$ grows exponentially with $-\Rea s$, i.e.,
there is always strong gradient from negative $\Rea s$ to positive. We
observe that for every zero~$s_0$ with $\Rea s_0 > c$ there is a curve of
constant phase, $\Ima\log(Z_X(s)) = \phi$, which starts at $s=s_0$ and
crosses the line $s=c + i\R$ eventually. This behavior is reflected
in the Newton fractal of $Z_X(s)$: For every root~$s_0$ with $\Rea s_0 > c$
the basin of attraction of the Newton map intersects the line
$c + i\R$ at least once. In Figure~\ref{fig:newton4cylinder} we
illustrate this behavior for the simplest Schottky surface, i.e., for the
hyperbolic cylinder.

To locate all zeros~$s_0$ with $\Rea s_0 > c$ we therefore choose a
sufficiently dense set of starting values along the line $c + i\R$ and
iterate Newton's method until it converges or fails. The set of
zeros we find can contain duplicates, which we need to filter out, and
we may miss zeros if the set of starting values is not dense enough.

Unless the convergence rate is analyzed in detail, Newton's method cannot
determine the order of a zero. We could use the argument principle in
a small neighborhood of a zero to find the order. However, except for
highly symmetric surfaces most of the resonances appear to be simple
zeros. For topological zeros on the real line the order is known.

\section{Examples}

We now illustrate the performance of the domain-refined Lagrange--Chebyshev
approximation with plots of the resonance spectra of various Schottky surfaces.

\subsection{Three-funnel surfaces}

\begin{figure}
  \begin{tabular}{ccc}
    \includegraphics[width=0.3\linewidth]{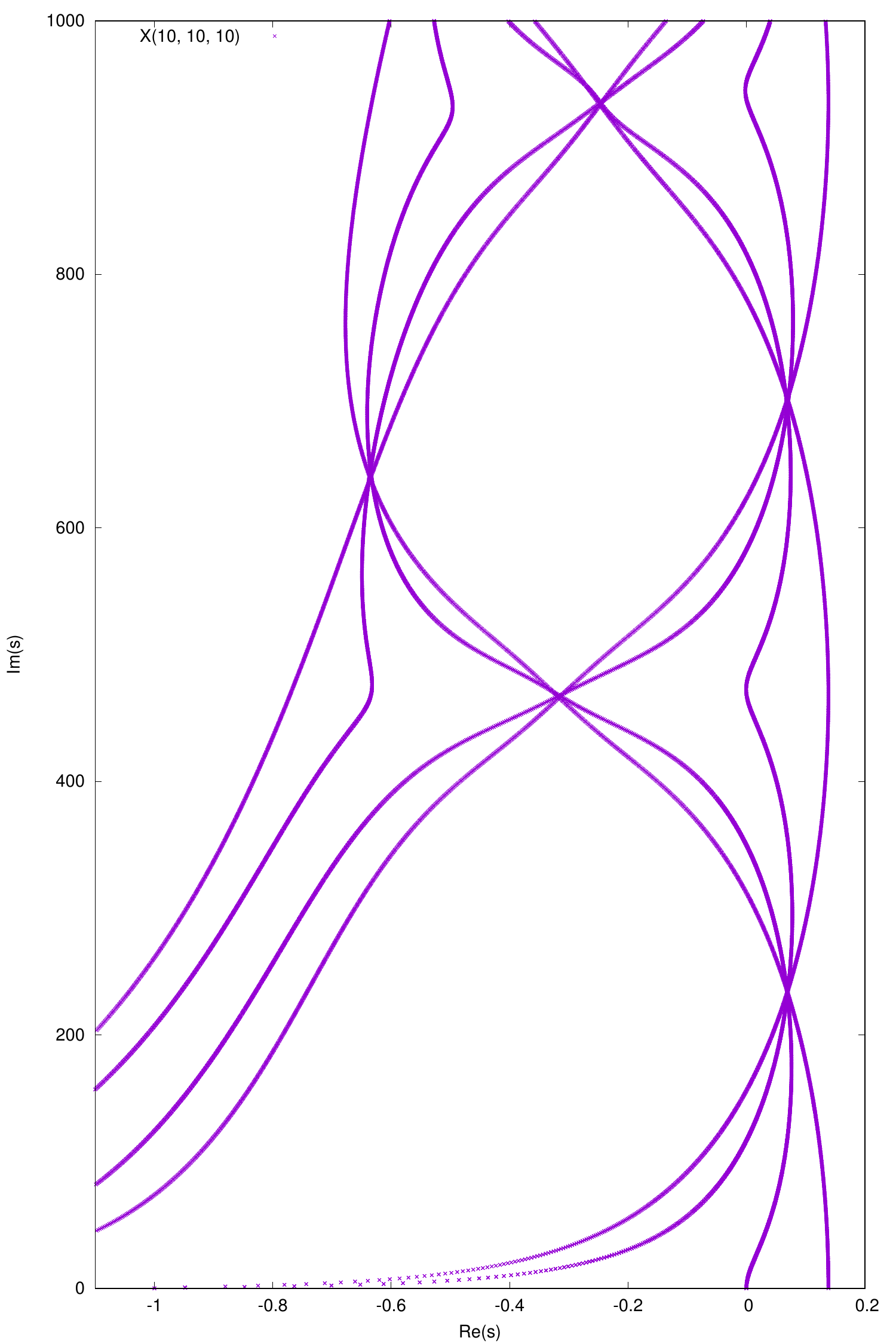} &
    \includegraphics[width=0.3\linewidth]{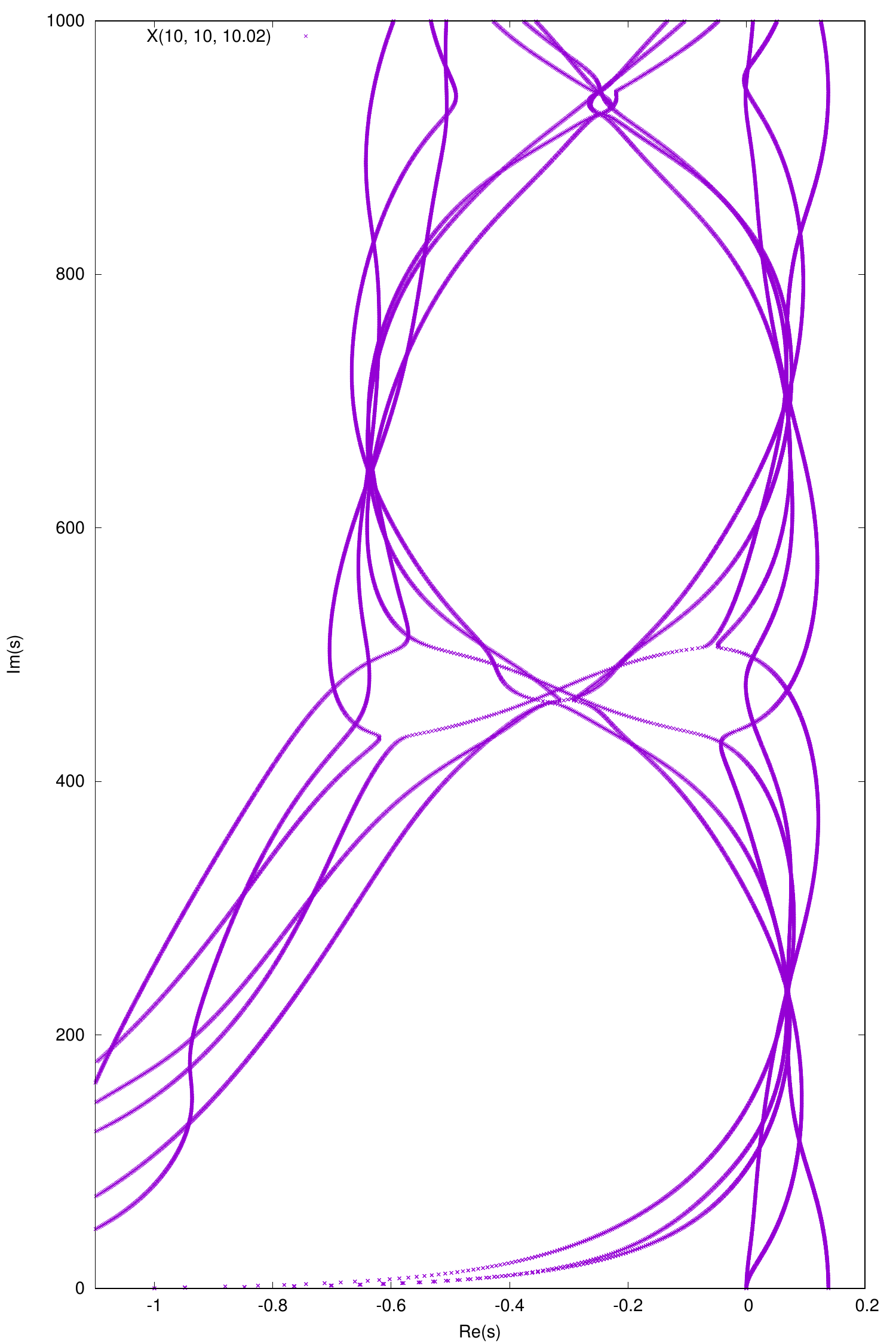} &
    \includegraphics[width=0.3\linewidth]{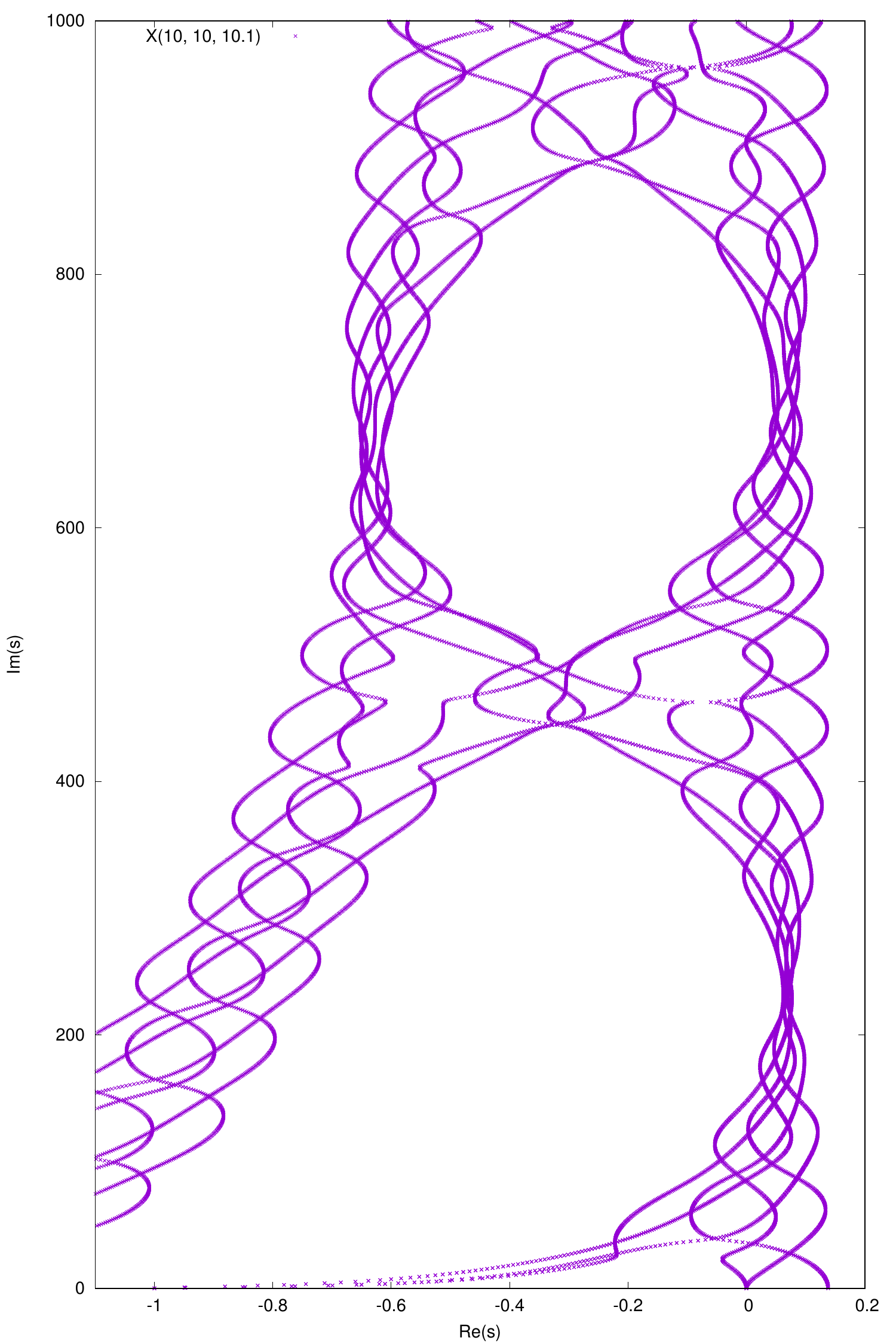}\\
    \includegraphics[width=0.3\linewidth]{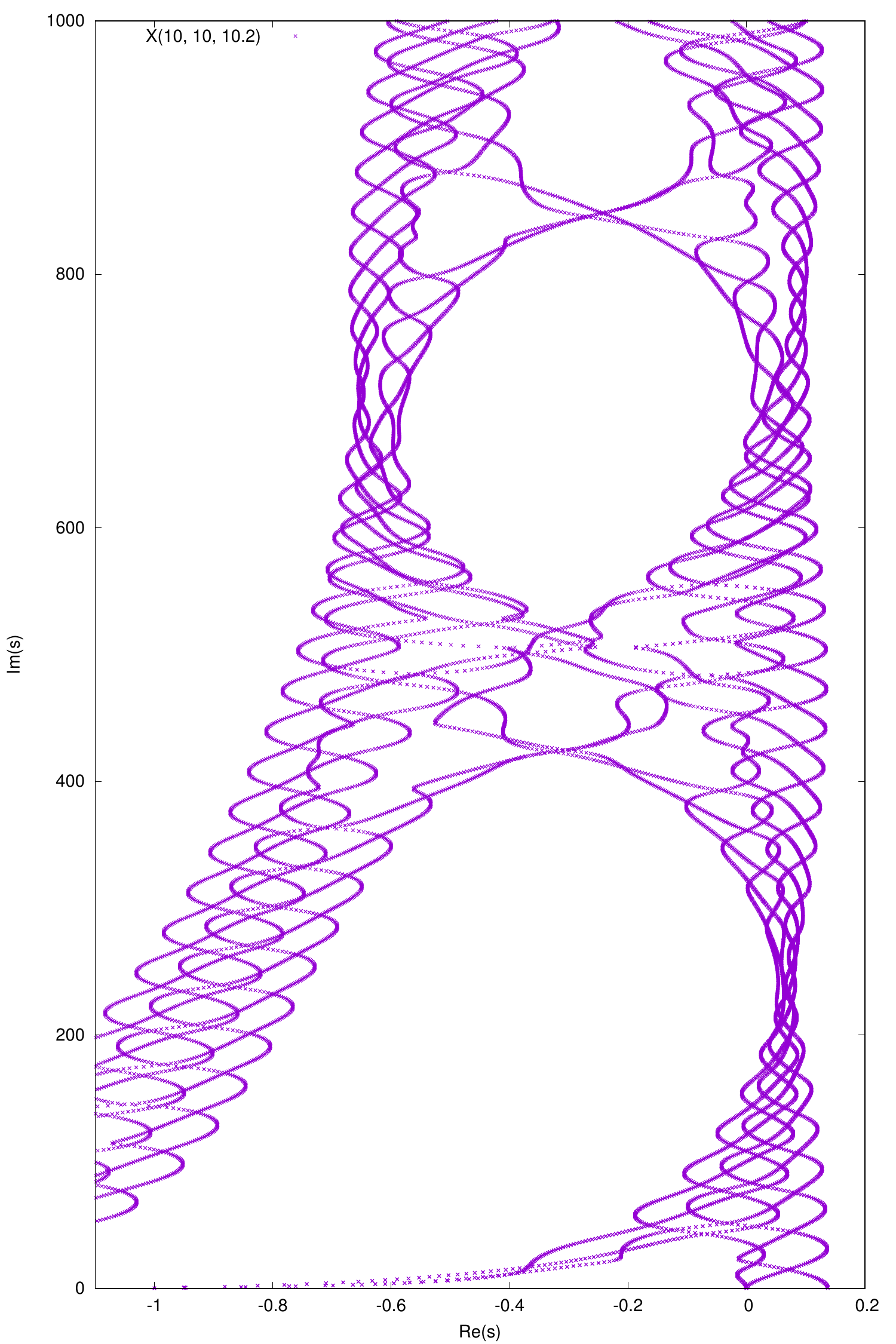} &
    \includegraphics[width=0.3\linewidth]{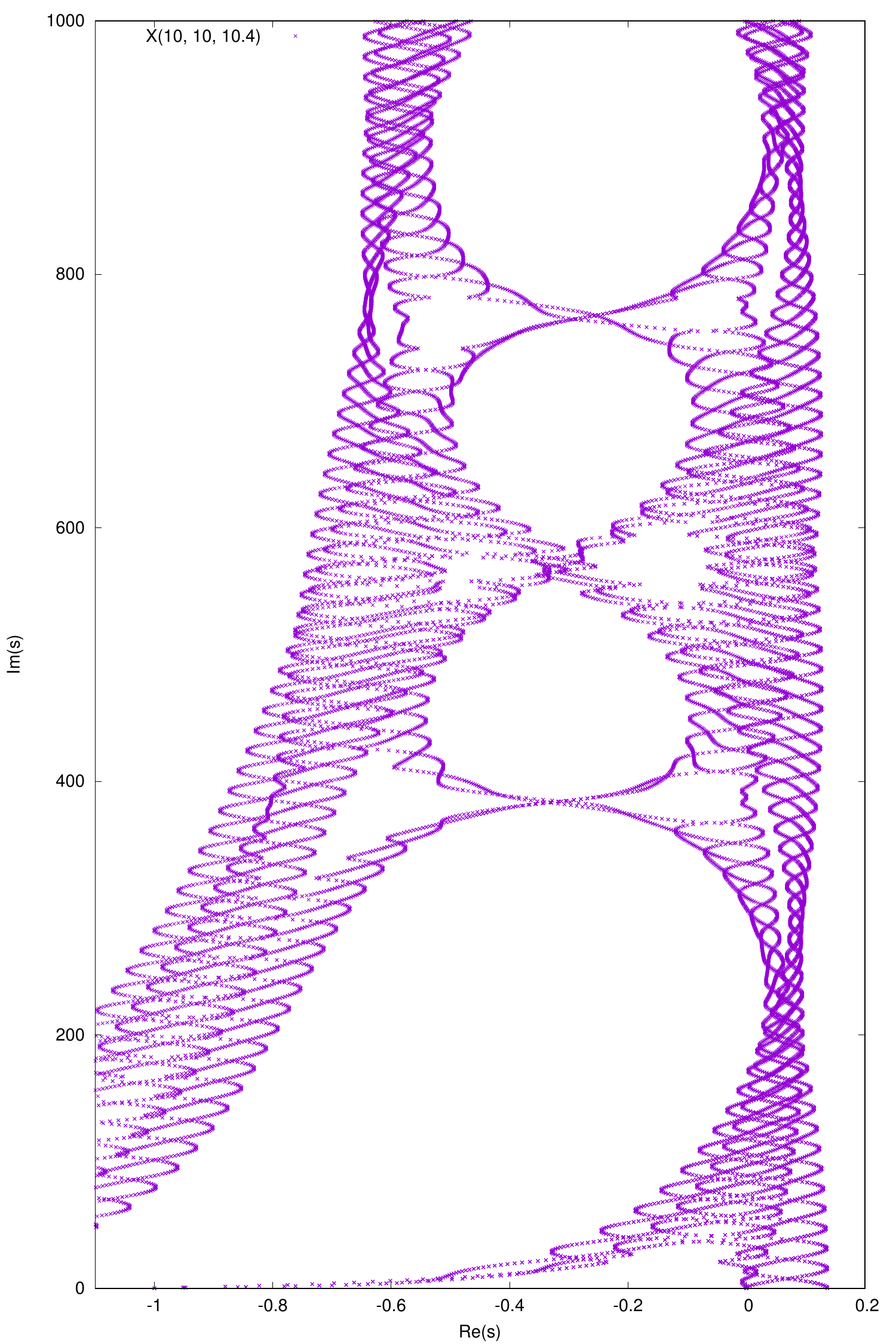} &
    \includegraphics[width=0.3\linewidth]{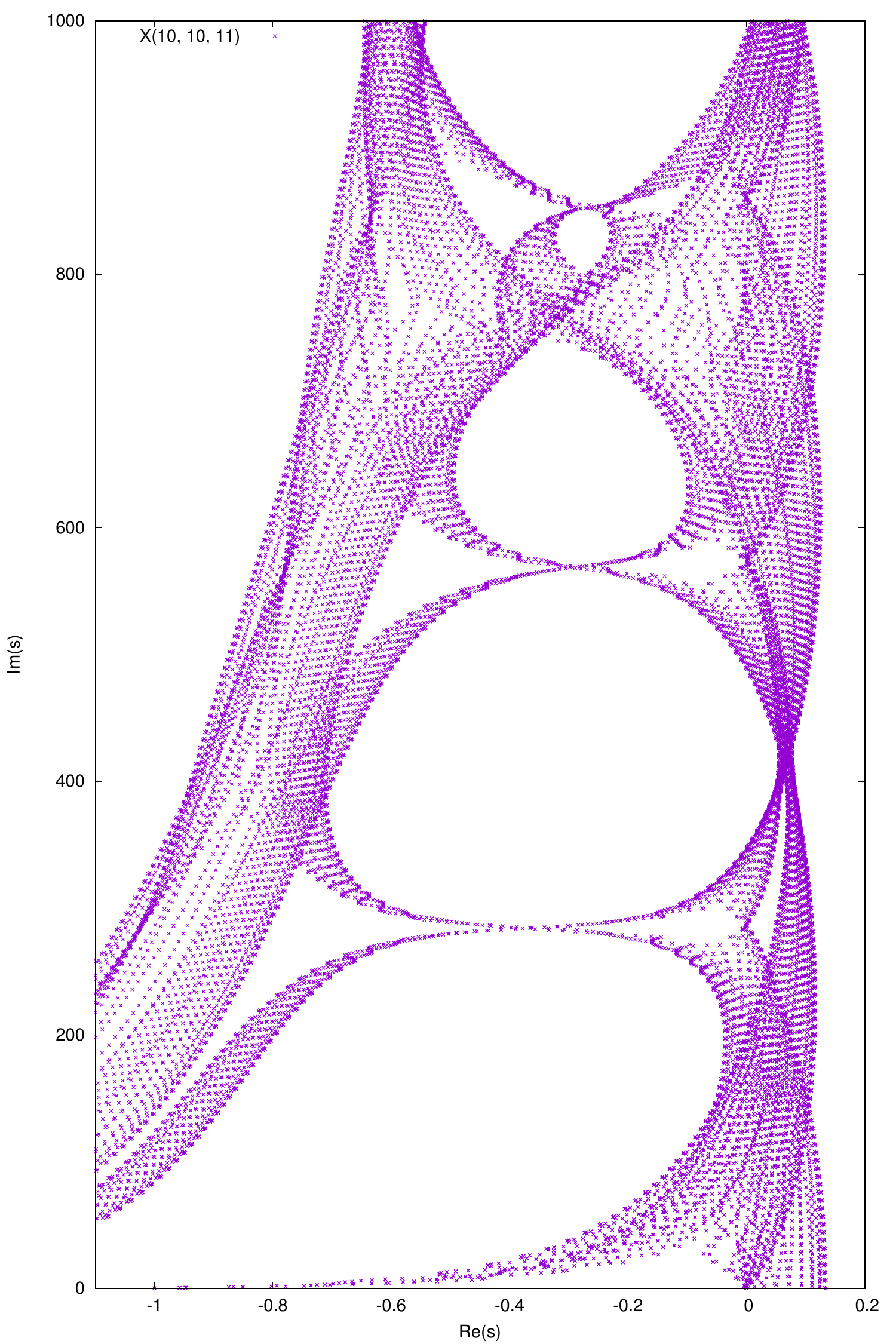}
  \end{tabular}
  \caption{Evolution of the resonance spectrum for the three-funnel
    surface $X(10,10,\ell)$ for $\ell=10, 10.02, 10.1, 10.2, 10.4$ and
    $11$.  }\label{fig:X1010ell}
\end{figure}

We start with three-funnel surfaces which were studied already in
Borthwick's seminal paper~\cite{Borthwick_experimental} and later
discussed in his book~\cite{BorthwickBook2nd}. These can be generated
using matrices of the form
\begin{equation}\label{def:S_of_ella}
  S(l,a) = \begin{pmatrix}
    \cosh(l/2) & a \sinh(l/2)\\
    a^{-1} \sinh(l/2) & \cosh(l/2)
  \end{pmatrix}
  \quad\text{with}\quad
  S(l,a)^{-1} = S(l,-a)\,.
\end{equation}
If we define the boundary of a disk $D$ by the condition
$S(l,a)'(z) = 1$, its center is at $-a\coth(l/2)$ and the radius
equals $|a/\sinh(l/2)|$. The transformation $z\to S(l,a).z$ maps
the exterior of $D$ to the inside of a disk $\tilde D$ of the same
radius but centered at $+a\coth(l/2)$. The geodesic between the
boundaries of the two disks has length $l$.

Using two generators of this type,
\begin{equation}
  S_1 = S(\ell_1,1)
  \quad\text{and}\quad
  S_2 = S(\ell_2,a_{2})\,,
\end{equation}
we glue the corresponding pairs of disks, $D_1,D_{-1}$ and $D_2,D_{-2}$,
to obtain a surface with three funnels. See Figure~\ref{fig:c3fun}. The first two funnels have
widths $\ell_1$ and $\ell_2$, the widths of the third, $\ell_3$, depends on the
parameter $a_{2}$ through the condition
\begin{equation}\label{eq:cond_on_a}
  \tr(S_1 S_{-2}) = -2 \cosh(\ell_3/2)\,,
\end{equation}
which ensures that the geodesics between $D_1$ and $D_2$ and between
$D_{-1}$ and $D_{-2}$ each have length $\ell_3/2$. Solving~\eqref{eq:cond_on_a} for
$a_{2}$ we define a function $A(\ell_1,\ell_2,\ell_3)$ for later use and obtain
\begin{equation}
  \begin{aligned}\label{def:a_of_ells}
    a_{2} & = A(\ell_1,\ell_2,\ell_3)  := d - \sqrt{d^2 - 1}\quad\text{with}\\
    d & = \frac{\cosh(\ell_1/2)\cosh(\ell_2/2) + \cosh(\ell_3/2)}{\sinh(\ell_1/2)\sinh(\ell_2/2)}\,.
  \end{aligned}
\end{equation}
Following Borthwick, we denote such a surface by $X(\ell_1,\ell_2,\ell_3)$.
The disks $D_{\pm 1}$ and $D_{\pm 2}$ form the starting point for domain-refinement, as illustrated
in Figure~\ref{fig:refinement}.

\begin{figure}
  {\centering\includegraphics[width=\linewidth]{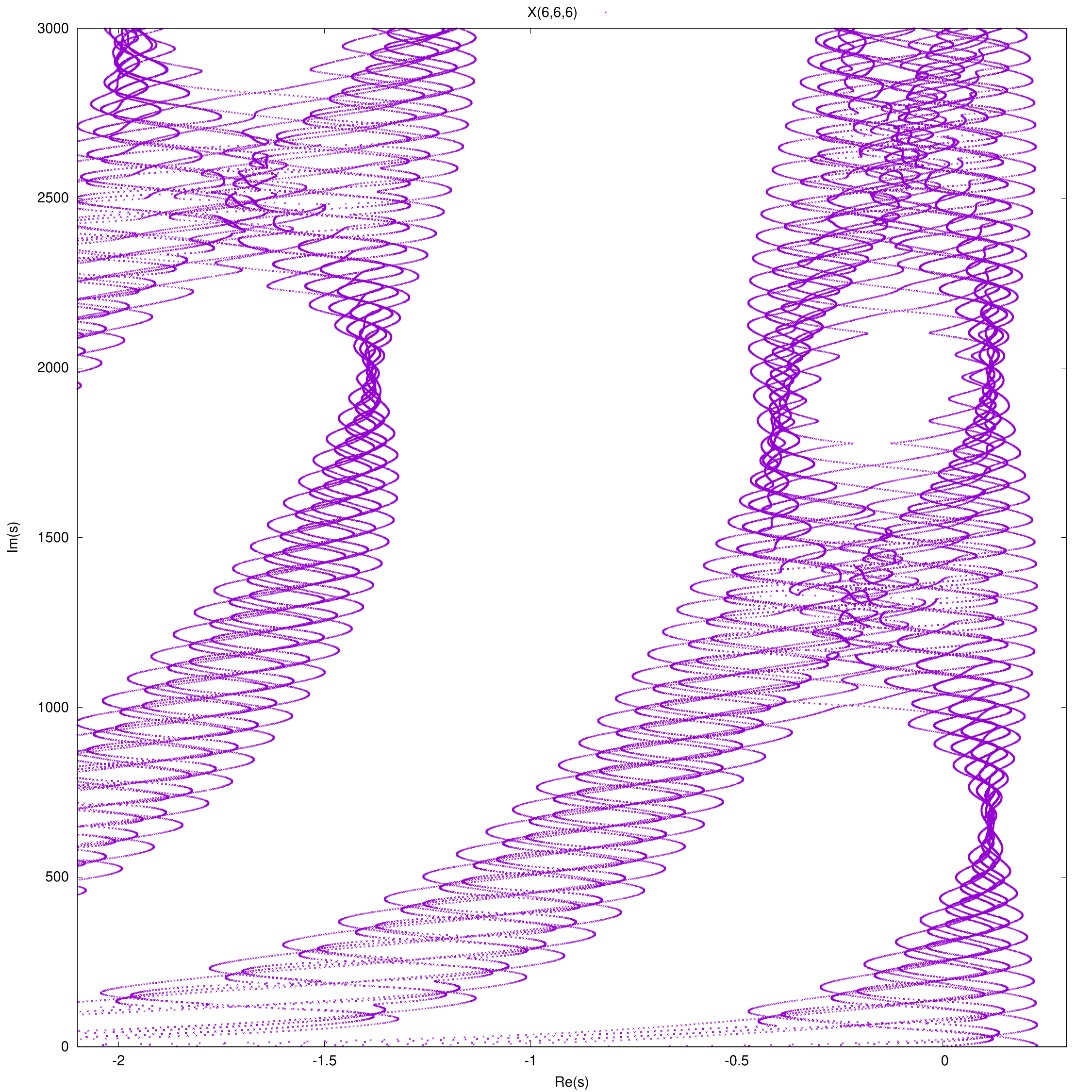}}
  \caption{Resonance spectrum for the three-funnel
    surface $X(6,6,6)$.}\label{fig:X666}
\end{figure}

As a first test, in Figure~\ref{fig:X1010ell} we reconsider the series
of resonance spectra for $X(10,10,\ell)$ first studied by
Borthwick~\cite{Borthwick_experimental}.  However, we extend the range
of the spectral parameter from $\Rea s \ge 0$ to $\Rea s>-1.1$. This
reveals many new and fascinating structures. For instance, the chains
of resonances seem to oscillate between strands formed by multiple
chains. Also, it looks as if the chains start at the topological zeros
$s = -\mathbb{N}$ and the Hausdorff dimension~$\delta$ of the limit set~$\Lambda(X)$ of~$X$. These features become even more apparent in
Figure~\ref{fig:X666} where we study the spectrum of $X(6,6,6)$ and
extend the range of the spectral parameter even further to
$\Rea s>-2.1$.

An aspect which did not receive much attention in previous studies is
the multiplicity of the resonances. We observe that for symmetric
three-funnel surfaces some of the zeros of $Z_X(s)$ have multiplicity
2. In Figure~\ref{fig:X777phase} we zoom in on the spectrum of the
surface $X(7,7,7)$ and show a phase plot of the function $Z_X(s)$ with
its argument coded in color. In each group of four zeros there are two
where the color wheel is circled twice. This corresponds to
multiplicity two. Interestingly, the Venkov--Zograf factorization used
by Borthwick--Weich does not split up these two-fold zeros,
compare~\cite[Figure~9]{Borthwick_Weich}. If the symmetry if broken, the
two-fold zeros split up and all resonances have multiplicity one, see
Figure~\ref{fig:X1010ell}.

\begin{figure}
  {\centering\includegraphics[width=\linewidth]{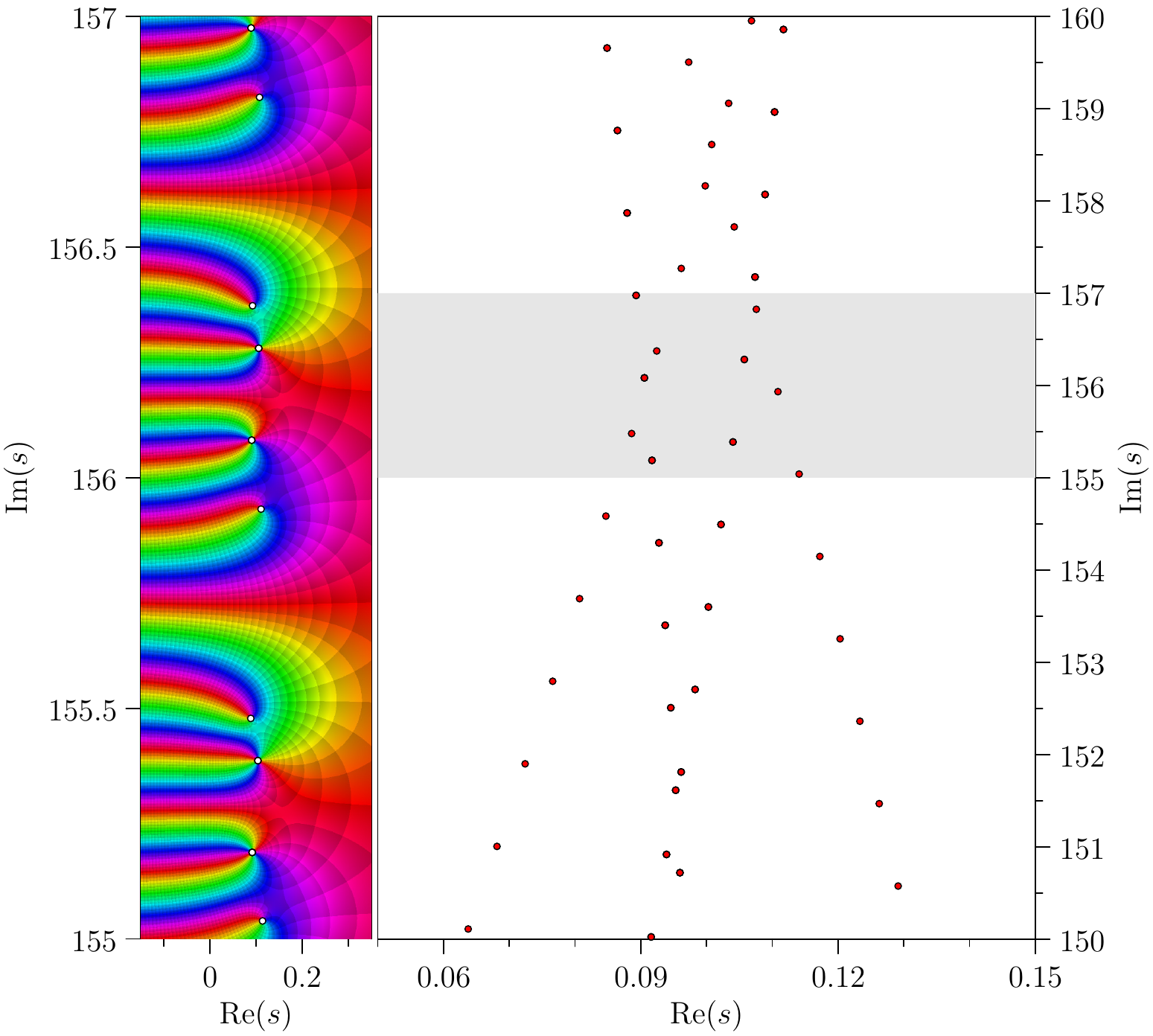}}
  \caption{Right: Zoom in on the resonance spectrum of the three-funnel surface
    $X(7,7,7)$. Left: Phase plot of $Z_X(s)$ with all zeros from the
    shaded area, indicating that some have multiplicity two. }\label{fig:X777phase}
\end{figure}

\subsection{$n$-funnel surfaces}

For surfaces with more than three funnels of widths
$\ell_1,\dots,\ell_n$, we extend the above construction by adding
further generators $S(l_k,a_k)$ of the type~\eqref{def:S_of_ella} and
corresponding pairs of disks~$D_{\pm k}$. In total we need $n-1$
generators (and their inverses) for $n$ funnels. The $l$-parameters of
the first and last generator determine the funnel-width of the first
and last funnel, $l_1=\ell_1$ and $l_{n-1}=\ell_n$, the $l$'s of the
other generators are free parameters describing a waist
circumference. The widths of the intermediate funnels can be fixed
using the function $A$ defined in~\eqref{def:a_of_ells} above. For
given $l_k$ we build up the parameters $a_k$ multiplicatively,
\begin{equation}
  a_1 = 1\,,\quad a_{k+1} = A(l_k,l_{k+1},\ell_{k+1})\ a_k\,.
\end{equation}
To make the surface symmetric, we need to tune the ``inner''
$l$-parameters such that equivalent waist circumferences are equal.
This can be done numerically.

Figure~\ref{fig:constr4fun} shows this construction for a symmetric
four-funnel surface. Here the width of all four funnels is 3 and symmetry
requires $l_2=4.853373$ for the inner generator $S_2$. Only then the
lengths of the geodesics $D_1\leftrightarrow D_3$ and
$D_{-1}\leftrightarrow D_{-3}$ add up to $l_2$ and the two waists have
equal circumference (red dashed lines). For surfaces with more funnels
the construction is similar, but there are more waistlines to match.

\begin{figure}
  {\centering \includegraphics[height=3.5cm]{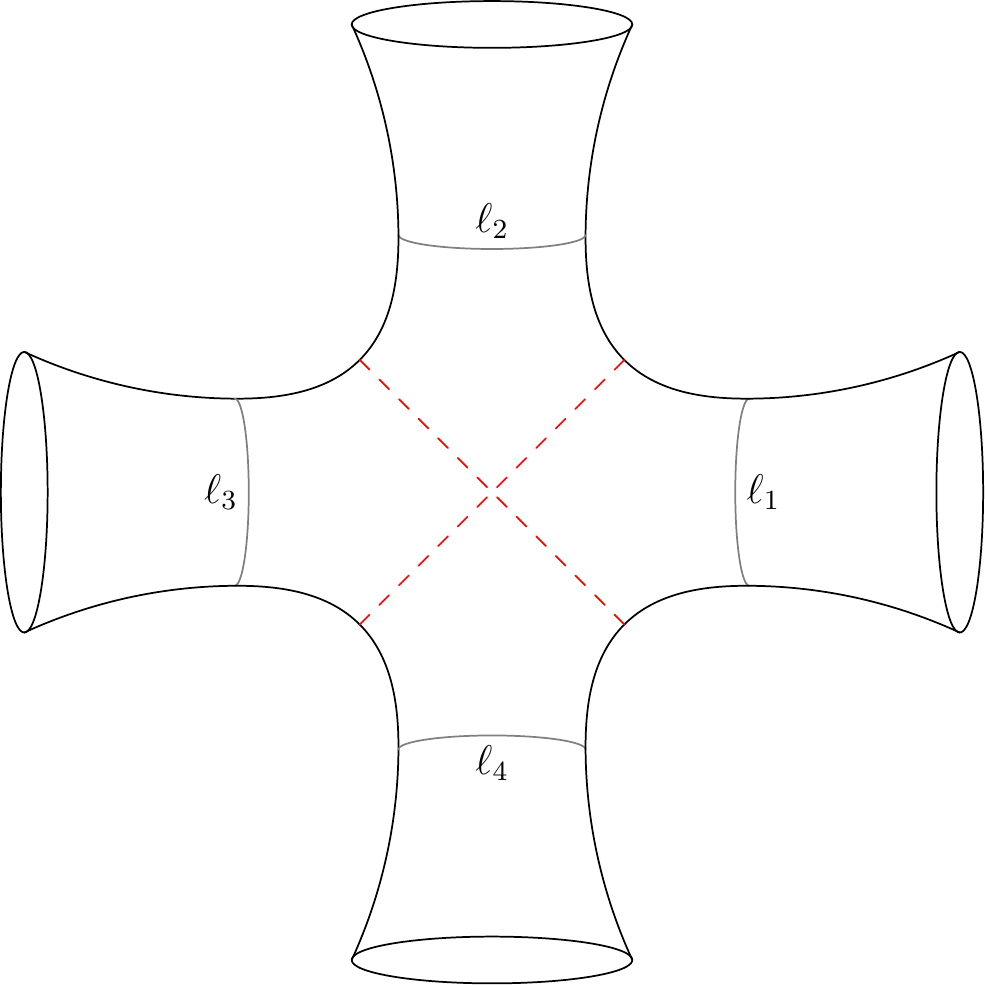}\hfill
    \includegraphics[height=3.5cm]{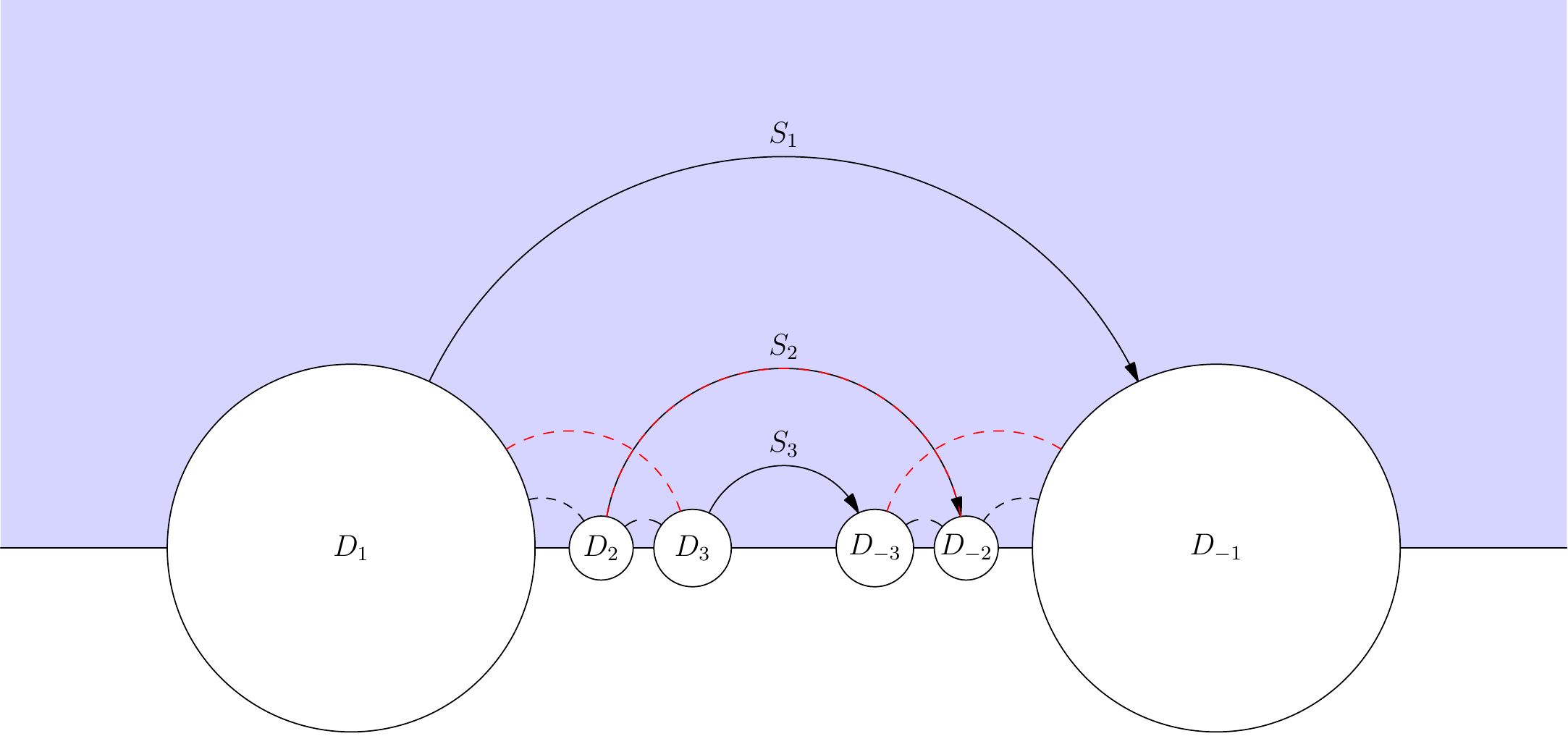} }
  \caption{Construction of a symmetric four-funnel
    surface.}\label{fig:constr4fun}
\end{figure}

\begin{figure}
  {\centering\includegraphics[width=0.48\linewidth]{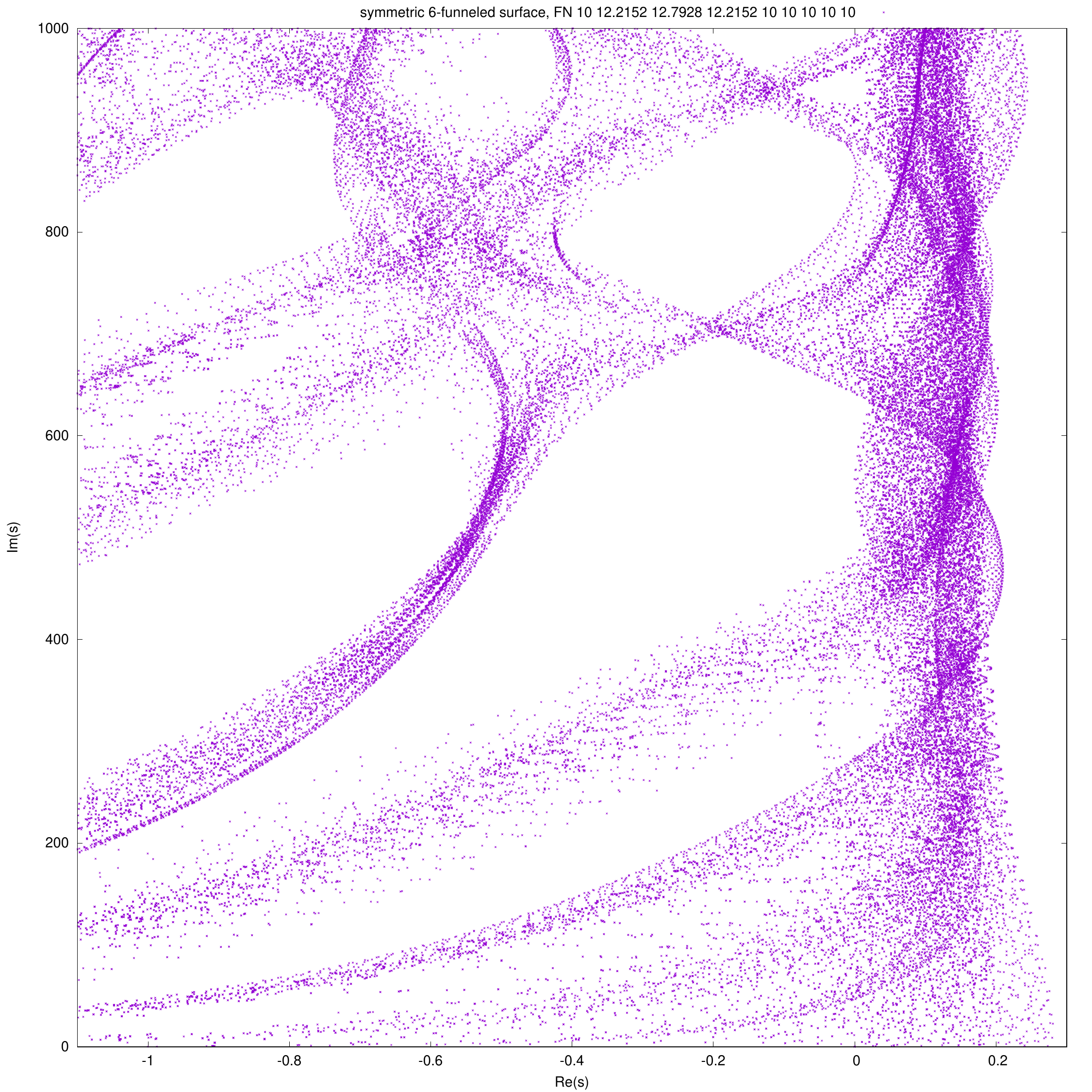}\hfill
    \includegraphics[width=0.48\linewidth]{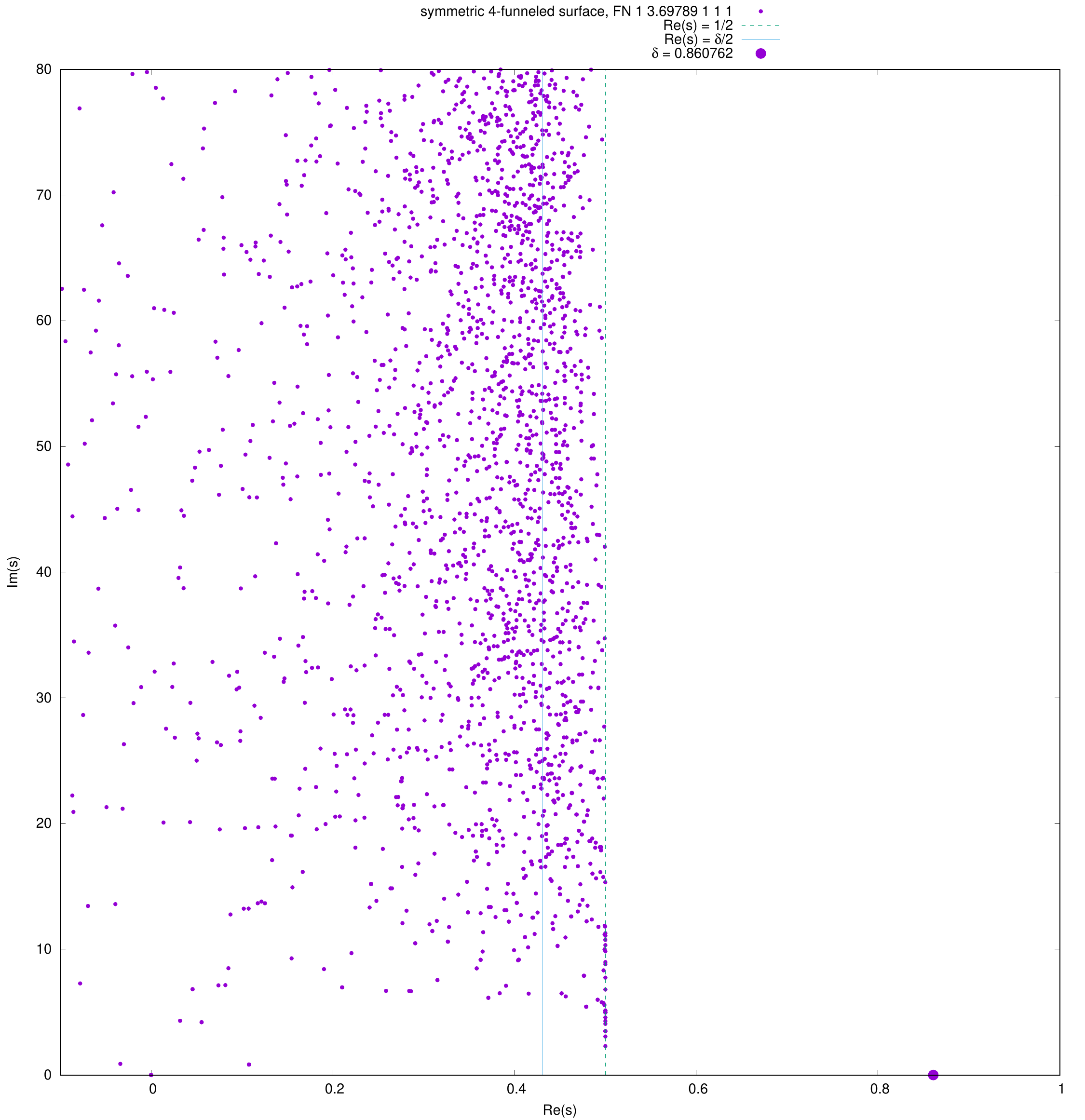}}
  \caption{Spectra of symmetric $n$-funnel surfaces. Left: Surface
    with 6 funnels all of width 10. Right: Surface with 4 funnels all
    of width 1. }\label{fig:nfunnel}
\end{figure}

Our construction has the advantage, that the size difference between
disks is not too pronounced and that the whole arrangement is more symmetric.
This makes it better suited for the Langrange--Chebyshev approach.

In Figure~\ref{fig:nfunnel} we show resonance spectra for symmetric
$n$-funnel surfaces. The data on the left is for a surface with 6
funnels all of width 10. Its construction requires 5 generators and
their inverses, corresponding to 5 pairs of disks. The data on the
right is for a surface with 4 funnels, but their width is only 1. For
such narrow funnels the Hausdorff dimension of the limit set is
approaching a value close to one, $\delta=0.86076$ in this case. The
data clearly shows that there are no resonances with $\Ima s \ne 0$ and
$\Rea s> 1/2$. Instead, some resonances accumulate close to the line
$\Rea s = 1/2$, indicated by the green dashed line. The data also seems
to support conjectures that the density of resonances has a
maximum near $\Rea s = \delta/2$ indicated by the blue solid line in the
figure~\cite{Jakobson_Naud}.

We note that small funnel widths or large $\delta$ are situations where
domain-refinement becomes important. The surface with 6 funnels of
width 10 was simulated with a subdivision level of $n=1$ only, and could
presumably be solved without domain refinement. The 4-funnel surface
with narrow funnels, on the other hand, required a subdivision level
of $n=3$. We will discuss such convergence properties in more detail in
a forthcoming article.

\subsection{Funneled tori}

\begin{figure}
  {\centering\includegraphics[width=\linewidth]{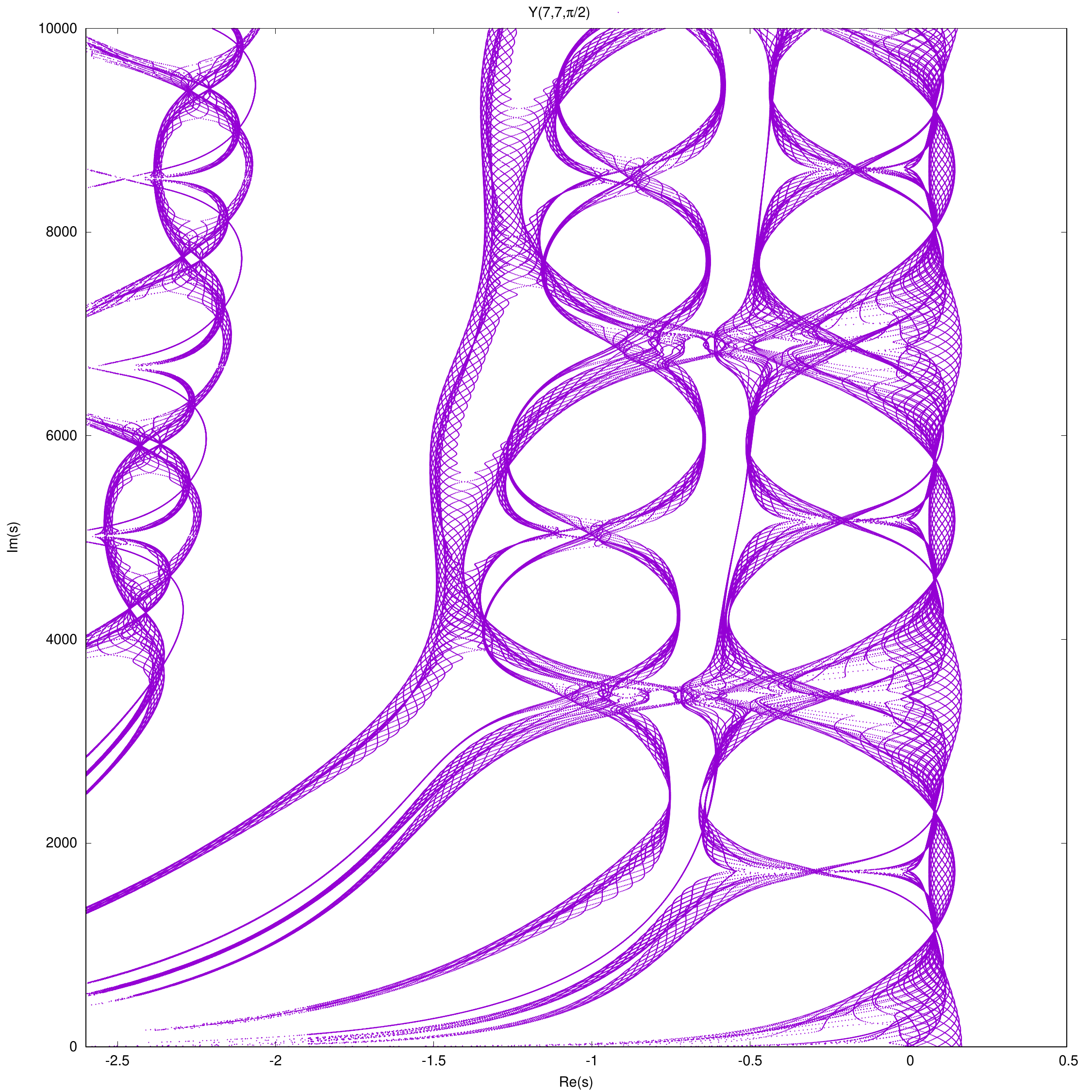}}
  \caption{Resonance spectrum for the funneled torus $Y(7,7,\pi/2)$.}\label{fig:Y772}
\end{figure}

As examples of surfaces of genus 1, we reconsider tori with one attached funnel. These were
also first studied by Borthwick~\cite{Borthwick_experimental} and constructed from the two
generators
\begin{equation}
  \begin{aligned}
    \tilde S_1 & = \begin{pmatrix}
      e^{\ell_1/2} & 0\\
      0 & e^{-\ell_1/2}
    \end{pmatrix}
    \,,\\
    \tilde S_2 & = \begin{pmatrix}
      \cosh(\ell_2/2)-\cos(\phi)\sinh(\ell_2/2) & \sin^2(\phi)\sinh(\ell_2/2) \\
      \sinh(\ell_2/2) &  \cosh(\ell_2/2)+\cos(\phi)\sinh(\ell_2/2)
    \end{pmatrix}\,.
  \end{aligned}
\end{equation}
For our algorithm this choice is not usable, since one of the
corresponding disks includes the point at infinity. We can avoid this
problem by rotating the above generators by $\pi/8$, i.e., we
use
\begin{equation}
  S_1 = R(\tfrac{\pi}{8})\, \tilde S_1\, R(-\tfrac{\pi}{8})\,,\quad
  S_2 = R(\tfrac{\pi}{8})\, \tilde S_2\, R(-\tfrac{\pi}{8})\,,\quad
  R(\psi) = \begin{pmatrix}
    \cos(\psi) & -\sin(\psi)\\
    \sin(\psi) & \cos(\psi)
  \end{pmatrix}\,.
\end{equation}
Our convention $S'(z) = 1$ then results in disks $D_{\pm 1}$,
$D_{\pm 2}$, which for $\phi=\pi/2$ have equal size, and small size
differences otherwise. In case that the $\ell$'s are too small, this
construction can lead to invalid, overlapping disks, and some other
choice is required. Following Borthwick, we denote such surfaces by
$Y(\ell_1,\ell_2,\phi)$. In Figure~\ref{fig:funtorus} we illustrate
the construction of $Y(4,4,\pi/2)$.

The resonance spectrum of funneled tori is more delicate with many
fine details, in particular, when the surface is symmetric. In
Figure~\ref{fig:Y772} we show the spectrum of $Y(7,7,\pi/2)$. Compared
to previous studies, we were able to increase the range of the spectral
parameter to $\Rea s>-2.6$.

\section{Algorithms}

In this section we list the algorithms for the calculation of resonances.

\begin{algorithm}[H]
  \caption{Construct $N\ttimes N$ sub-block $\TO_s^{(N)}$ of full transfer operator.}\label{alg:TO_block}
  \begin{algorithmic}[1]
    %%%%% function for one block of Ls
    \Function{LBlock}{$N,I_v,I_w,g$}
    \State $c_v,r_v \gets \text{center \& radius of } I_v$
    \State $c_w,r_w \gets \text{center \& radius of } I_w$
    \For{$i=1,\dots,N$}
    \State $ x_i\gets \cos\big(\frac{\pi}{N}(i - \frac{1}{2})\big)$
    \EndFor
    \For{$i=1,\dots,N$}
    \For{$j=1,\dots,N$}
    \State $\begin{aligned}  m_{ij} & \gets K_N[(g(c_v+r_v x_i) - c_w)/r_w, x_j]\\
      f_i & \gets g'(c_v+r_v x_i)
    \end{aligned}$
    \EndFor
    \EndFor
    \State \textbf{return} $[f,m]$
    \Comment{$f:$ $N$-vector, $m:$ $N\ttimes N$-matrix}
    \EndFunction
  \end{algorithmic}
\end{algorithm}

\clearpage

\begin{algorithm}[H]
  \caption{Construct index set $\mc W_n$.}\label{alg:indexset}
  \begin{algorithmic}[1]
    %%%%% function to construct the set W
    \Function{IndexSet}{$q,n$}
    \State $I_G \gets [-q,\dots,-1,1,\dots,q]$
    \State $r\gets \begin{cases}
      1 & n<2\\ -1 & \text{otherwise}
    \end{cases}$
    \If{$n=0$}
    \State $\mc W_0\gets [\, [k] \mid k\in I_G]$
    % \For{$k\in I_G$}
    % \State $\mc W_n\gets \Call{Append}{\mc W_n,[k]}$
    % \EndFor
    \Else
    \State $\mc W_n\gets [\,]$
    \For{$w\in\mc W_{n-1}=\Call{IndexSet}{q,n-1}$}
    \For{$k\in I_G$}
    \If{$r\cdot k\ne \Call{First}{w}$}
    \State $\tilde w \gets \Call{Prepend}{w, k}$
    \State $\mc W_n \gets \Call{Append}{\mc W_n, \tilde w}$
    \EndIf
    \EndFor
    \EndFor
    \EndIf
    \State \textbf{return} $\mc W_n$
    \Comment{$2q(2q-1)^n$-element array of $(n+1)$-element index sets.}
    \EndFunction
  \end{algorithmic}
\end{algorithm}

\begin{algorithm}[H]
  \caption{Build domain-refined transfer operator $\TO_s^{(N)}$.}\label{alg:TO_subdiv}
  \begin{algorithmic}[1]
    %%%%% function preparing all blocks
    \Statex \Function{LParts}{$N,q,n,[D_{-q},\dots,D_{-1},D_1,\dots,D_q]$}
    \State $I_G \gets [-q,\dots,-1,1,\dots,q]$
    \State $S_k \gets \text{generator mapping outside of }D_k\text{ to inside of }D_{-k}\,\forall k\in I_G$
    \State $I_k \gets D_k\cap\mathbb{R}\,\forall k\in I_G$
    \State $\mc W_n \gets \Call{IndexSet}{q,n}$
    \For{$v=(v_1,\dots,v_n,\ell_v)\in \mc W_n$}
    \For{$w=(w_1,\dots,w_n,\ell_w)\in \mc W_n$}
    \If{$w=(w_1,v_1,\dots,v_{n-1},-v_n)$}
    \State $I_v\gets S_{v_1}\cdots S_{v_n}.I_{\ell_v}$
    \State $I_w\gets S_{w_1}\cdots S_{w_n}.I_{\ell_w}$
    \State $[F_{vw},M_{vw}] \gets \Call{LBlock}{n,I_v,I_w,S_{-w_1}}$
    \Else
    \State $[F_{vw},M_{vw}] \gets [1,0]$
    \EndIf
    \EndFor
    \EndFor
    \State \textbf{return} $[F,M]$
    \Comment {\parbox{0.5\linewidth}{$F:$ $d\ttimes d$ array of $N$-vectors\\
        $M:$ $d\ttimes d$ array of $N\ttimes N$-matrices\\
        $d = \dim(\mc W_n) = 2q(2q-1)^n$}}
    \EndFunction
  \end{algorithmic}
\end{algorithm}

\begin{algorithm}[H]
  \caption{Evaluate $Z(s) = \det(1-\TO_s^{(N)})$ using pre-calculated parts.}\label{alg:zeta_of_s}
  \begin{algorithmic}[1]
    \Function{Z}{$s,N,q,n,[F,M]$}
    \State $\mc W_n \gets \Call{IndexSet}{q,n}$
    \State $d\gets \dim(\mc W_n)$
    \For{$v\in \mc W_n$}
    \For{$w\in \mc W_n$}
    \State $L_{vw} \gets M_{vw}\diag(F_{vw}^{s})$
    \Comment {\parbox{0.5\linewidth}{$F_{ij}$ is $N$-vector; exponentiate elements\\
        $L_{vw}, M_{vw}$ are $N\ttimes N$ matrices}}
    \EndFor
    \EndFor
    \State \textbf{return} $\det\left(1 - \begin{bmatrix}
        L_{11} & \cdots & L_{1d}\\
        \vdots & & \vdots\\
        L_{d1} & \cdots & L_{dd}
      \end{bmatrix}
    \right)$
    \Comment{\parbox{0.3\linewidth}{call library for $\det()$\\
        ($LU$ factorization)}}
    \EndFunction
  \end{algorithmic}
\end{algorithm}

\section*{Acknowledgements} 

OFB was supported by the EPSRC grant EP/R012008/1.
AP acknowledges support by the DFG grants PO~1483/2-1 and PO~1483/2-2 and she wishes to thank the Max Planck Institute for Mathematics in Bonn for excellent working conditions during large parts of the research periods for this project. Further, OFB and AP wish to thank the Hausdorff Institute for Mathematics in Bonn for financial support and excellent working conditions during the HIM trimester program ``Dynamics: Topology and Numbers,'' where they prepared parts of this manuscript.

\bibliography{resonances}
\bibliographystyle{resonances2}

\setlength{\parindent}{0pt}

\end{document}